\documentclass[11pt,reqno]{amsart}
\usepackage{amsmath,amstext,amssymb,amscd,hyperref}
\usepackage{verbatim}
\usepackage{enumerate}
\usepackage{mathrsfs}
\usepackage[dvipsnames,usenames]{xcolor}

\usepackage{amsthm}

\newtheorem{theorem}{Theorem}[section]
\newtheorem{lemma}[theorem]{Lemma}

\theoremstyle{definition}
\newtheorem{definition}[theorem]{Definition}

\theoremstyle{remark}
\newtheorem{remark}[theorem]{Remark}

\numberwithin{equation}{section}

\def\bar{\overline}

\author[C. Cao]{Chongsheng Cao}
\address{Department of Mathematics \& Statistics \\Florida International University\\
Miami, Florida 33199, USA} \email{caoc@fiu.edu}

\author[Y. Guo]{Yanqiu Guo}
\address{Department of Mathematics \& Statistics \\Florida International University\\
Miami, Florida 33199, USA} \email{yanguo@fiu.edu}

\author[E. S. Titi]{Edriss S. Titi}
\address{Department of
Mathematics \\ Texas A\&M University\\
College Station, TX 77843, USA \textbf{AND} Department of Computer Science and Applied Mathematics \\ Weizmann Institute of Science \\ Rehovot 7610001 Israel}
 \email{titi@math.tamu.edu, edriss.titi@weizmann.ac.il}

\title[Regularity for a rapidly rotating constrained convection model]
{Global regularity for a rapidly rotating constrained convection model of tall columnar structure with weak dissipation}
\date{October 8, 2018}

\begin{document}

\maketitle

\begin{abstract}
We study a three-dimensional fluid model describing rapidly rotating convection that takes place in tall columnar structures.
The purpose of this model is to investigate the cyclonic and anticyclonic coherent structures. Global existence, uniqueness, continuous dependence on initial data, and large-time behavior of strong solutions are shown provided the model is regularized by a weak dissipation term.
\end{abstract}

\section {Introduction} \label{S1}
\subsection{The model}
For the purpose of investigating the cyclonic and anticyclonic coherent structure in the Rayleigh-B\'enard convection under the influence of a rapid rotation, Sprague et al. \cite{SJKW} (see also Julien et al. \cite{JKMW,JK}) introduced and simulated the following asymptotically reduced system for rotationally constrained convection that takes place in a tall column:
\begin{align}
&\frac{\partial w}{\partial t} + \mathbf u \cdot \nabla_h w - \frac{\partial \phi}{\partial z} = \Gamma \theta + \frac{1}{Re} \Delta_h w,   \label{mo-1} \\
&\frac{\partial \omega}{\partial t} + \mathbf u \cdot \nabla_h \omega - \frac{\partial w}{\partial z} = \frac{1}{Re} \Delta_h \omega ,   \label{mo-2}\\
&\frac{\partial \theta}{\partial t}  + \mathbf u \cdot \nabla_h \theta +  w \bar{w\theta} = \frac{1}{Pe} \Delta_h \theta,   \label{mo-3}\\
&\nabla_h \cdot \mathbf u=0.   \label{mo-4}
\end{align}
In the above system, the velocity vector field $(u,v,w)^{tr}$ is defined in a three-dimensional periodic domain $\Omega= [0,L]^2 \times [0,1]$,
where $\mathbf u=(u,v)^{tr}$ denotes the horizontal component of the velocity vector field. The unknown $\theta$ represents the fluctuation of the temperature such that the horizontal spatial mean $\bar \theta(z)=0$, for every $z\in [0,1]$. For a function $f$ defined in $\Omega$, the notation $\bar f$ stands for the horizontal mean
\begin{align*}
\bar f(z) = \frac{1}{L^2} \int_{[0,L]^2} f(x,y,z) dx dy,  \;\;  \text{for}   \;\;  z\in [0,1].
\end{align*}
We denote the horizontal gradient by $\nabla_h=(\frac{\partial}{\partial x}, \frac{\partial}{\partial y})$ and denote the horizontal Laplacian by
$\Delta_h=\frac{\partial^2}{\partial x^2}+ \frac{\partial^2}{\partial y^2}$. The unknown $\omega=\nabla_h \times \mathbf u=\partial_x v - \partial_y u$
represents the vertical component of the vorticity. As usual, the horizontal stream function $\phi$ is defined as
$\phi= (-\Delta_h)^{-1} \omega$, with $\bar \phi=0.$ Also, a few dimensionless numbers appear in the model. Specifically, $Re$ is the Reynolds number, $\Gamma$ is the buoyancy number, and $Pe$ is the P\'eclet number.

System (\ref{mo-1})-(\ref{mo-4}) is an asymptotically reduced model derived from three-dimensional Boussinesq equations governing buoyancy-driven rotational flow in tall columnar structures, by assuming that the ratio of the depth of the fluid layer to the horizontal scale is large, and the angular velocity is fast.

The global regularity for system (\ref{mo-1})-(\ref{mo-4}) is unknown.
The main difficulty of analyzing (\ref{mo-1})-(\ref{mo-4}) lies in the fact that the physical domain is three-dimensional, whereas the regularizing viscosity acts only on the horizontal variables, and the equations contain troublesome terms $\frac{\partial \phi}{\partial z}$ and $\frac{\partial w}{\partial z}$ involving the derivative in the vertical direction.

In this work, we regularize the convection model (\ref{mo-1})-(\ref{mo-4}) by imposing a very weak vertical dissipation term $\epsilon^2 \frac{\partial^2\phi}{\partial z^2}$ to the vorticity equation (\ref{mo-2}), namely, we consider the regularized system
\begin{align}
&\frac{\partial w}{\partial t} + \mathbf u \cdot \nabla_h w - \frac{\partial \phi}{\partial z} = \Gamma \theta + \frac{1}{Re} \Delta_h w,   \label{model-1} \\
&\frac{\partial \omega}{\partial t} + \mathbf u \cdot \nabla_h \omega - \frac{\partial w}{\partial z} = \frac{1}{Re} \Delta_h \omega +  \epsilon^2 \frac{\partial^2\phi}{\partial z^2},   \label{model-2}\\
&\frac{\partial \theta}{\partial t} + \mathbf u \cdot \nabla_h \theta +  w \bar{w\theta} = \frac{1}{Pe} \Delta_h \theta,   \label{model-3}\\
&\nabla_h \cdot \mathbf u=0.   \label{model-4}
\end{align}
The main goal of this paper is to prove the global regularity of system (\ref{model-1})-(\ref{model-4}). We remark that, as a dissipation,
$\epsilon^2 \frac{\partial^2\phi}{\partial z^2}$ is much weaker than the vertical viscosity
$\epsilon^2 \frac{\partial^2\omega}{\partial z^2}$, since $\omega= -\Delta_h \phi$. The purpose of introducing and analyzing (\ref{model-1})-(\ref{model-4}) is to shed some light on the global regularity problem for the 3D rotationally constrained convection model (\ref{mo-1})-(\ref{mo-4}), a subject of future investigation. Notably, there is no physical meaning for the dissipation term $\epsilon^2 \frac{\partial^2\phi}{\partial z^2}$, however, it can be viewed as a numerical dissipation.

In order to prove the existence of strong solutions for (\ref{mo-1})-(\ref{mo-4}), we introduce a ``Galerkin-like" approximation scheme. In fact, the Galerkin-like system consists of a system of ODEs coupled with a PDE, and it is set up in the format of an iteration. This special Galerkin scheme represents a ``novelty" of the paper.

It is worth mentioning that the three-dimensional Hasagawa-Mima equations \cite{HM, HM2}, describing plasma turbulence, share a comparable structure with the convection model (\ref{mo-1})-(\ref{mo-4}). Although the well-posedness problem for the 3D inviscid Hasagawa-Mima equations is still unsolved, in a recent work \cite{CGT} we established the global well-posedness of strong solutions for a Hasegawa-Mima model with partial dissipation. Also, Cao et al. \cite{CFT} showed the global well-posedness for an inviscid pseudo-Hasegawa-Mima model in three dimensions.

The paper is organized as follows. For the rest of section \ref{S1}, we introduce suitable function spaces for solutions and provide some identities related to the nonlinearities of model (\ref{model-1})-(\ref{model-4}). Then we state the main results, namely, the existence, uniqueness, continuous dependence on initial data, and large-time behavior of strong solutions to (\ref{model-1})-(\ref{model-4}). Section \ref{S2} features some inequalities which are essential for our analysis.
In section \ref{sec-exist}, we prove the existence of strong solutions by using a Galerkin-like approximation method. In section \ref{sec-unique}, we justify the uniqueness of strong solutions and the continuous dependence on initial data. Finally, we study the large-time behavior of solutions in section \ref{sec-large}.

\vspace{0.1 in}

\subsection{Preliminaries}
Let $\Omega=[0,L]^2 \times [0,1]$ be a three-dimensional fundamental periodic domain. The standard $L^p(\Omega)$ norm for periodic functions is denoted by $\|f\|_p=\left(\int_{\Omega} |f|^p dx dy dz\right)^{\frac{1}{p}}$, $p\geq 1$. As usual, the $L^2(\Omega)$ inner product of real-valued periodic functions $f$ and $g$ is defined by
$(f,g)=\int_{\Omega} fg dx dy dz$. Also $H^s(\Omega)$, $s\geq 0$, denotes the standard Sobolev spaces for periodic functions. In addition, we define a Hilbert space
\begin{align*}
H^1_h(\Omega)= \left\{f\in L^2(\Omega):   \int_{\Omega} |\nabla_h f|^2 dx dy dz <\infty \right\},
\end{align*}
endowed with an inner product $(f,g)_{H^1_h(\Omega)}  =   \int_{\Omega} \nabla_h f  \cdot  \nabla_h g dx dy dz $.

Let $f\in H^1_h(\Omega)$ with zero horizontal mean, i.e. $\bar f=0$, then the Poincar\'e inequality holds:
\begin{align}   \label{poin}
\|f\|_2^2 \leq \gamma \|\nabla_h f\|_2^2, \;\;  \text{where}    \;\;   \gamma= L^2/(4\pi^2).
\end{align}

For sufficiently smooth periodic functions $\mathbf u$, $f$ and $g$, such that $\nabla_h \cdot \mathbf{u}=0$, an integration by parts shows
\begin{align}   \label{iden-0}
\left(\mathbf{u} \cdot \nabla_h f,g\right)=- \left(\mathbf{u} \cdot \nabla_h g,f\right)
\end{align}
This implies
\begin{align}   \label{iden-1}
\left(\mathbf{u} \cdot \nabla_h f,f\right)=0.
\end{align}
Note that the horizontal velocity $\mathbf u$, the vertical vorticity $\omega$, and the horizontal stream function $\phi$ have the following relations:
\begin{align}   \label{iden-4}
\omega=\nabla_h \times \mathbf u=v_x-u_y, \;\;\; \omega=-\Delta_h \phi,\;\;\;  \mathbf u = (\phi_y,-\phi_x)^{tr}.
\end{align}
It follows that
\begin{align}  \label{iden-3}
(\omega,\phi)=\|\mathbf u\|_2^2.
\end{align}
Also, $\|\omega\|_2  \approx \|\nabla_h \mathbf u\|_2$. In addition, by (\ref{iden-0}) and (\ref{iden-4}), we have
\begin{align}   \label{iden-2}
\left(\mathbf u \cdot \nabla_h  f, \phi\right) = -\left(\mathbf u \cdot \nabla_h  \phi, f\right)=0,
\end{align}
for sufficiently regular functions $\mathbf u$, $\phi$ and $f$ such that $\mathbf u = (\phi_y,-\phi_x)^{tr}$.

We remark that, since $\nabla_h \cdot \mathbf u =0$ and $\omega =   \nabla_h \times \mathbf u = v_x  -  u_y$, then $u=(-\Delta_h)^{-1} \omega_y$ and $v= \Delta_h^{-1} \omega_x$, if $\bar{\mathbf u}=0$. Thus, the horizontal velocity $\mathbf u$ and the vertical component $\omega$ of the vorticity determine each other uniquely, provided $\nabla_h \cdot \mathbf u =0$ and $\bar{\mathbf u}=0$.

\vspace{0.1 in}

\subsection{Main results}
Before stating our main results, we shall give a precise definition of strong solutions for system (\ref{model-1})-(\ref{model-4}).
Let us first introduce a suitable function space for strong solutions to (\ref{model-1})-(\ref{model-4}).
Specifically, we define the following space of periodic functions:
\begin{align}   \label{def-V}
V = \big\{(\mathbf u,w,\theta)^{tr} \in (H^1(\Omega))^4:  \nabla_h \cdot \mathbf u =0, \;\;  \bar{\mathbf u}=0,  \;\; \bar w=\bar \theta=0 \big\}.
\end{align}

\begin{definition}  \label{def-sol}
Let $T>0$. Assume $(\mathbf u_0, w_0, \theta_0)^{tr} \in V$ and $\omega_0 =   \nabla_h \times \mathbf u_0 \in L^2(\Omega)$. We call $(\mathbf u, w, \theta)^{tr}\in V$
with $\omega \in L^2(\Omega)$ a strong solution for system (\ref{model-1})-(\ref{model-4})
on $[0,T]$ if
\begin{enumerate}[(i)]
\item  $\mathbf u$, $w$, $\theta$ and $\omega$ have the following regularity:
\begin{align}   \label{main-2}
\begin{cases}
\mathbf u, w, \theta\in L^{\infty}(0,T;H^1(\Omega))\cap C([0,T];L^2(\Omega));\\
\omega \in L^{\infty}(0,T; L^2(\Omega)) \cap  C([0,T]; (H^1_h(\Omega))'  )        ;\\
\nabla_h \omega, \Delta_h w, \Delta_h \theta \in L^2(\Omega \times (0,T));   \\
\omega_z, \nabla_h w_z, \nabla_h \theta_z, \phi_{zz} \in L^2(\Omega \times (0,T));   \\
 \mathbf u_t, w_t, \theta_t \in L^2(\Omega \times (0,T));  \\
 \omega_t \in   L^2(0,T;(H^1_h(\Omega))').
 \end{cases}
\end{align}

\item equations (\ref{model-1})-(\ref{model-3}) hold in the following function spaces respectively:
\begin{align}     \label{main-1}
\begin{cases}
w_t + \mathbf u \cdot \nabla_h w - \phi_z= \Gamma \theta + \frac{1}{Re} \Delta_h w,     \;\;\text{in}\;\; L^2(\Omega \times (0,T));\\
\omega_t + \mathbf u \cdot \nabla_h \omega - w_z = \frac{1}{Re} \Delta_h \omega  + \epsilon^2  \phi_{zz} ,
\;\;\text{in}\;\; L^2(0,T;(H^1_h(\Omega))');\\
\theta_t  + \mathbf u \cdot \nabla_h \theta +  w \bar{w\theta} = \frac{1}{Pe} \Delta_h \theta,   \;\; \text{in} \;\; L^2(\Omega \times (0,T)),
\end{cases}
\end{align}
such that $\nabla_h \cdot \mathbf u=0$, $\omega = \nabla_h \times \mathbf u = - \Delta_h \phi$, with $\bar \phi=0$.

\item $\mathbf u(0)=\mathbf u_0$, $w(0)=w_0$, $\omega(0)=\omega_0$, $\theta(0)=\theta_0$.

\end{enumerate}
\end{definition}

\vspace{0.02 in}

Now we are ready to state the main results of the manuscript. Our first theorem is concerned with the global existence and uniqueness of strong solutions as well as the continuous dependence on initial data.
\begin{theorem}  \label{main}
Let $T>0$. Assume initial data $(\mathbf u_0, w_0, \theta_0)^{tr} \in V$ and $\omega_0 = \nabla_h \times \mathbf u_0 \in L^2(\Omega)$. Then system (\ref{model-1})-(\ref{model-4}) admits a unique strong solution $(\mathbf u, w, \theta)\in V$ with $\omega \in L^2(\Omega)$ on $[0,T]$ in the sense of Definition \ref{def-sol}. Moreover, the energy identity is valid for every $t\in [0,T]$:
\begin{align}    \label{energy}
&\frac{1}{2}\left(\|w(t)\|_2^2 + \|\mathbf u(t)\|_2^2 +  \|\theta(t)\|_2^2 \right) + \int_0^t \|\bar{w\theta}\|_2^2 ds \notag\\
&+\int_0^t \left[  \frac{1}{Re}    \left(\|\nabla_h w\|_2^2 + \|\nabla_h \mathbf u\|_2^2 \right) + \frac{1}{Pe} \|\nabla_h \theta\|_2^2 + \epsilon^2 \|\phi_z\|_2^2 \right]ds\notag\\
& =\frac{1}{2}\left(\|w_0\|_2^2 + \|\mathbf u_0\|_2^2 +  \|\theta_0\|_2^2 \right) + \Gamma \int_0^t (\theta, w) ds.
\end{align}
Also, if $(\mathbf u_0^n,w_0^n, \theta_0^n)^{tr}$ is a bounded sequence of initial data in $V$ and $\omega_0^n =  \nabla_h \times \mathbf u_0^n$ is a bounded sequence in $L^2(\Omega)$ such that $(\mathbf u_0^n,w_0^n,\theta_0^n)^{tr}$ converges to $(\mathbf u_0,w_0,\theta_0)^{tr}$ with respect to the $L^2(\Omega)$-norm
and $\omega_0^n$ converges to $\omega_0$ in $(H^1_h(\Omega))'$, then the corresponding strong solution $(\mathbf u^n,w^n,\theta^n)^{tr}$ converges to $(\mathbf u,w,\theta)^{tr}$ in $C([0,T];(L^2(\Omega))^4)$, and $\omega^n$ converges to $\omega$ in $C([0,T];(H^1_h(\Omega))')$.
\end{theorem}

The next result is concerned with the large-time behavior of strong solutions for system (\ref{model-1})-(\ref{model-4}).
It shows the energy decays to zero exponentially in time. Also, the $L^2$-norm of $\omega$ as well as the $L^2$-norm of the horizontal gradient of $w$ and $\theta$ approach zero exponentially fast.
However, the $L^2$-norm of the vertical derivative of $(\mathbf u, w,\theta)^{tr}$ grows at most exponentially in time. Recall we have defined $\gamma= L^2/(4\pi^2)$ in (\ref{poin}).
\begin{theorem}     \label{thm-decay}
Let $\kappa \geq 1$ such that $Pe \not = 2\kappa Re$.
Assume $(\mathbf u, w, \theta)^{tr}\in V$ with $\omega \in L^2(\Omega)$ is a global strong solution for system  (\ref{model-1})-(\ref{model-4}) in the sense of Definition \ref{def-sol}.
Then, for all $t \geq 0$,
\begin{align}
&\|\theta(t)\|_2^2  \leq e^{-\frac{2}{\gamma Pe}t} \|\theta_0\|_2^2,     \label{thm2-00}   \\
&\|\mathbf u(t)\|_2^2 + \|w(t)\|_2^2   \leq e^{-\frac{1}{\kappa \gamma Re}t}  \left(  \|\mathbf u_0\|_2^2 +  \|w_0\|_2^2  \right)
+C   \left( e^{-\frac{2}{\gamma Pe}t}   +  e^{-\frac{1}{\kappa \gamma Re}t} \right)\|\theta_0\|_2^2.     \label{thm2-01}
\end{align}
In addition, for large $t$,
\begin{align}  \label{thm2-1}
&\|\omega(t)\|_2^2 + \|\nabla_h w(t)\|_2^2 + \|\nabla_h \theta(t)\|_2^2    \notag\\
&\leq \left(e^{-\frac{2}{\gamma Pe}t} + e^{-\frac{1}{\kappa \gamma Re}t} \right)  C(\|\mathbf u_0\|_2,  \|w_0\|_2,  \|\theta_0\|_2, \|\partial_z \theta_0\|).
\end{align}
Moreover, for all $t\geq 0$,
\begin{align}    \label{thm2-2}
 \|  \mathbf u_z(t)\|_2^2 + \|w_z(t)\|_2^2  +  \|  \theta_z(t)\|_2^2  \leq C(\|\mathbf u_0\|_{H^1}, \|w_0\|_{H^1}, \|\theta_0\|_{H^1})
 e^{C\left(\|\theta_0\|_2^2+  \|\partial_z \theta_0\|_2^2  +1 \right)t}.
\end{align}
\end{theorem}

\vspace{0.1 in}

\section{Auxiliary inequalities}  \label{S2}
In this section, we provide some inequalities which are essential for analyzing model (\ref{model-1})-(\ref{model-4}). The first one is an anisotropic Ladyzhenskaya inequality which will be used repeatedly in this manuscript.
\begin{lemma}   \label{lemma1}
Let $f\in H^1(\Omega)$, $g\in H^1_h(\Omega)$ and $h\in L^2(\Omega)$. Then
\begin{align*}
\int_{\Omega} |fgh| dx dy dz  \leq C \left(\|f\|_2 + \|\nabla_h f\|_2 \right)^{\frac{1}{2}}  \left(\|f\|_2 + \|f_z\|_2 \right)^{\frac{1}{2}}
 \|g\|_2^{\frac{1}{2}} \left(\|g\|_2  + \|\nabla_h g\|_2 \right)^{\frac{1}{2}} \|h\|_2.
\end{align*}
\end{lemma}
We have proved Lemma \ref{lemma1} in \cite{CGT}. Also, a similar inequality can be found in \cite{CT11}.

The next inequality is derived from the Agmon's inequality.
\begin{lemma}    \label{lemma2}
Assume $f$, $f_z \in L^2(\Omega)$, then
\begin{align}   \label{ineq-2}
\sup_{z\in [0,1]} \int_{[0,L]^2} |f(x,y,z)|^2 dx dy \leq C\|f\|_2   \left(\|f\|_2+\|f_z\|_2\right).
\end{align}
\end{lemma}
\begin{proof}
Recall the Agmon's inequality in one dimension (cf. \cite{Agmon, CF88, Temam}):
$$\|\psi\|_{L^{\infty}([0,1])} \leq C\|\psi\|_{L^2([0,1])}^{1/2}  \|\psi\|_{H^1([0,1])}^{1/2},   \;\; \text{for any} \;\; \psi \in H^1([0,1]).$$
Then, we have, for a.e. $z\in [0,1]$,
\begin{align*}
&\int_{[0,L]^2} |f(x,y,z)|^2 dx dy \notag\\
&\leq C\int_{[0, L]^2}       \left(\int_0^1   |f|^2 dz  \right)^{1/2}      \left(\int_0^1 (|f|^2 + |f_z|^2) dz \right)^{1/2}   dx dy    \notag\\
&\leq C \|f\|_2 \left( \|f\|_2 + \|f_z\|_2 \right),
 \end{align*}
where we have used the Cauchy-Schwarz inequality in the last step.
\end{proof}

\vspace{0.1 in}

Next, we derive an elementary Gronwall-type inequality which will be used to deal with the temperature equation (\ref{model-3}) in section \ref{sec-z}.
\begin{lemma}   \label{lemma-3}
Given continuous functions $\eta$, $f$, $g$, $h: [0,\infty) \rightarrow [0,\infty)$ such that $\eta\in C^1([0,\infty))$.
Suppose
\begin{align}     \label{ein-1}
\eta \eta' + h \leq f + g\eta,  \;\;\; \text{on} \;\;  [0,\infty),
\end{align}
then
\begin{align*}
\eta^2(t)  +   2\int_0^t h(s) ds \leq C\left[  \eta^2(0)  +  \int_0^t   f(s) ds  +   \left(\int_0^t  g(s) ds   \right)^2   \right],   \;\;\text{for all}\;\;  t\geq 0.
\end{align*}
\end{lemma}

\begin{proof}
Set $\xi(t)=\eta(t) -  \int_0^t    g(s) ds$. By (\ref{ein-1}), $\eta \eta' \leq f + g\eta$ on $[0,\infty)$, i.e.,
\begin{align}  \label{ein-2}
\xi \xi' +  \xi' \int_0^t g(s) ds  \leq f,        \;\;\; \text{on} \;\;  [0,\infty).
\end{align}
On the one hand, if $\xi'\geq 0$ on some interval $[a,b]\subset [0,\infty)$, then (\ref{ein-2}) implies that $\xi \xi' \leq f$ on $[a,b]$, namely,
\begin{align}   \label{ein-3}
\xi^2(t) \leq \xi^2(a)  + 2\int_{a}^t  f(s) ds,    \;\; \text{for all} \;\;  t\in [a,b].
\end{align}
On the other hand, if $\xi' \leq 0$ on some interval $[a,b]\subset [0,\infty)$, then $\xi(t) \leq \xi(a)$ for all $t\in [a,b]$, so (\ref{ein-3}) also holds. In sum, we have
\begin{align}   \label{ein-4}
\xi^2(t) \leq \xi^2(0)  + 2\int_{0}^t  f(s) ds  = \eta^2(0)  + 2\int_{0}^t  f(s) ds,        \;\; \text{for all} \;\;  t\geq 0.
\end{align}
Since $\eta(t)=\xi(t) +  \int_0^t g(s) ds$, we obtain from (\ref{ein-4}) that
\begin{align}  \label{ein-5}
\eta^2(t)  &\leq 2 \xi^2(t) +   2 \left(\int_0^t g(s) ds\right)^2    \notag\\
&\leq   2  \eta^2(0)  + 4\int_{0}^t  f(s) ds  +     2 \left(\int_0^t g(s) ds\right)^2,   \;\;\text{for all}\;\; t\geq 0.
\end{align}

Next, we substitute (\ref{ein-5}) into the right-hand side of (\ref{ein-1}). Then
\begin{align}  \label{ein-6}
\frac{1}{2}[\eta^2(t)]' + h(t) \leq f(t) + g(t) \left [C  \left(\eta(0)  + \Large(\int_{0}^t  f(s) ds\Large)^{1/2}  +   \int_0^t g(s) ds \right)  \right],
\end{align}
for all $t\geq 0$. Integrating (\ref{ein-6}) over $[0,t]$ yields
\begin{align*}
\eta^2(t) +  2\int_0^t h(s) ds
&\leq  \eta^2(0) +  C \int_0^t  \left[  g(s) \left(\eta(0)  + \Large(\int_{0}^s  f(\tau) d\tau\Large)^{1/2}  +   \int_0^s g(\tau) d\tau \right)  \right]ds\notag\\
&\leq   \eta^2(0) +   C  \left(\eta(0)  + \Large(\int_{0}^t  f(\tau) d\tau\Large)^{1/2}  +   \int_0^t g(\tau) d\tau \right)   \int_0^t g(s)ds  \notag\\
&\leq C\left[\eta^2(0)  +\int_{0}^t  f(s) ds    +    \left( \int_0^t g(s) ds \right)^2     \right],
\end{align*}
where in the last step we use Young's inequality.
\end{proof}

\vspace{0.1 in}
Finally, we state a well-known uniform Gronwall Lemma. The proof can be found, e.g., in \cite{Temam}.
\begin{lemma}   \label{gronwall}
Let $g$, $h$, $\eta$ be three positive locally integrable functions on $[0,\infty)$ such that $\eta'$ is locally integrable, and which satisfy
\begin{align*}
\frac{d\eta}{dt} \leq g\eta +h,  \;\; \text{for}\;\; t\geq 0.
\end{align*}
Then,
\begin{align*}
\eta(t+1)\leq e^{\int_t^{t+1}  g(s) ds }  \left(    \int_t^{t+1} \eta(s) ds  + \int_t^{t+1}  h(s) ds     \right) ,  \;\; \text{for all} \;\; t\geq 0.
\end{align*}
\end{lemma}

\vspace{0.1 in}

\section{Existence of strong solutions}   \label{sec-exist}
Our strategy for proving the existence of strong solutions for system (\ref{model-1})-(\ref{model-4}) is a ``modified" Galerkin method.
The \emph{a priori} estimate for $\theta$ involves $L^{\infty}$ norm in the vertical variable, which is not easy to obtain with the standard Galerkin approximation scheme. To overcome the difficulty, we propose a ``Galerkin-like" system, which consists of a Galerkin approximation for velocity equations only, coupled with a PDE for the temperature.

\subsection{Galerkin-like approximation system}   \label{MGalerkin}
We assume initial data $\mathbf u_0, w_0, \theta_0 \in H^1(\Omega)$ with $\bar{\mathbf u_0} =0$, $\bar{w_0}=\bar{\theta_0}=0$, and $\omega_0 =   \nabla_h \times \mathbf u_0 \in L^2(\Omega)$.

Let $e_{\mathbf j} = \exp\left(2\pi i [(j_1 x+j_2 y)/L + j_3 z] \right)$ for $\mathbf j=(j_1,j_2,j_3)^{tr}\in \mathbb Z^3$, which form a basis of the $L^2(\Omega)$ space of periodic functions in $\Omega=   [0,L]^2 \times [0,1] $. For $m\in \mathbb N$, denote by $P_m\left(L^2(\Omega)\right)$ the subspace of $L^2(\Omega)$ spanned by $\{e_{\mathbf j}\}_{|\mathbf j|\leq m}$. Also, for an $L^2(\Omega)$ function $f=\sum_{\mathbf j\in \mathbb Z^3} f_{\mathbf j} e_{\mathbf j}$, where $f_{\mathbf j}=(f,e_{\mathbf j})$,
we denote by $P_m f=\sum_{|\mathbf j|\leq m} f_{\mathbf j} e_{\mathbf j}$ the orthogonal projection.

In order to prove the existence of strong solutions for system (\ref{model-1})-(\ref{model-4}), we introduce the following ``Galerkin-like" approximation system, for $m\geq 2$,
\begin{align}
&\partial_t w_m+ P_m\left(\mathbf u_m \cdot \nabla_h w_m\right) - \partial_z \phi_m= \Gamma P_m {\theta^{(m-1)}} + \frac{1}{Re} \Delta_h w_m, \label{Gl-1} \\
&\partial_t \omega_m + P_m\left(\mathbf u_m \cdot \nabla_h \omega_m\right) - \partial_z w_m = \frac{1}{Re} \Delta_h \omega_m  + \epsilon^2 \partial_{zz} \phi_m , \label{Gl-2}\\
&\nabla_h \cdot \mathbf u_m=0,  \label{Gl-3}    \\
&\omega_m = \nabla_h \times \mathbf u_m = -\Delta_h \phi_m,  \;  \text{with}\;  \bar{\phi_m}=0,   \label{Gl-4}  \\
&\partial_t {\theta^{(m)}}  + \mathbf u_m \cdot \nabla_h {\theta^{(m)}}  +  w_m \bar{w_m{\theta^{(m)}}}= \frac{1}{Pe} \Delta_h {\theta^{(m)}},
 \label{Gl-5}
\end{align}
where $\omega_m, \mathbf u_m, \phi_m, w_m \in  P_m\left(L^2(\Omega)\right)$, with the initial condition
\begin{align}   \label{Gl-6}
\omega_m(0)=P_m \omega_0,  \;\;   \mathbf u_m (0) = P_m \mathbf u_0,  \;\;  w_m(0)=P_m w_0 ,  \;\;  {\theta^{(m)}}(0)=\theta_0.
\end{align}
Since (\ref{Gl-1}) includes the term ${\theta^{(m-1)}}$, for $m\geq 2$, we have to specify $\theta^{(1)}$.
Here, we let $\theta^{(1)}$ satisfy the heat equation
\begin{align}   \label{theta1}
\partial_t \theta^{(1)} - \frac{1}{Pe}\Delta_h \theta^{(1)}=0, \;\;  \text{with}   \;\;   \theta^{(1)}(0)=\theta_0.
\end{align}

For a given $\theta^{(m-1)}$, velocity equations (\ref{Gl-1})-(\ref{Gl-4}) are genuine Galerkin approximation at level $m$, which can be regarded as a system of ODEs, whereas the temperature equation (\ref{Gl-5}) is a PDE. On the one hand, $\omega_m$, $\mathbf u_m$, $w_m$ are in the subspace $P_m (L^2(\Omega))$, namely, they are finite linear combination of Fourier modes. On the other hand, we do not demand $\theta^{(m)}$ to be finite combination of Fourier modes, so the superscript $m$ is adopted to emphasize the distinction between $\theta^{(m)}$ and $(\omega_m,\mathbf u_m,w_m)^{tr}$.

It is important to notice that the ``Galerkin-like" system (\ref{Gl-1})-(\ref{Gl-6}) is set up in the format of an iteration.
Let $T>0$. We claim, for any $m\geq 2$, system (\ref{Gl-1})-(\ref{Gl-6}) has a unique solution on $[0,T]$. This can be seen by an induction process described as follows.
To begin the induction, a function $\theta^{(1)} \in C([0,T];H^1(\Omega))$ is given satisfying (\ref{theta1}).
Now, we assume ${\theta^{(m-1)}} \in C([0,T];H^1(\Omega))$ is known for an $m\geq 2$, and show that system (\ref{Gl-1})-(\ref{Gl-6}) possesses a unique solution $(\omega_m, \mathbf u_m,  w_m,  {\theta^{(m)}})^{tr}$ on $[0,T]$.
Indeed, since the velocity equations (\ref{Gl-1})-(\ref{Gl-4}) form a system of ODEs with quadratic nonlinearities, by the standard theory of ordinary differential equations, a unique classical solution $(\omega_m,\mathbf u_m,  w_m)^{tr}$ for (\ref{Gl-1})-(\ref{Gl-4}) exists for a short time. Furthermore, one can show that the $L^2$ norm of $(\mathbf u_m,  w_m)^{tr}$ has a bound independent of time (see (\ref{L2-4}) below), thus $(\omega_m,\mathbf u_m,  w_m)^{tr}$ can be extended to $[0,T]$. Because $\mathbf u_m$ and $w_m$ have finitely many Fourier modes, they are analytic in space. Next we input $\mathbf u_m$ and $w_m$ into the temperature equation (\ref{Gl-5}) and solve for $\theta^{(m)}$. At this stage, $\mathbf u_m$ and $w_m$ are known smooth functions, thus (\ref{Gl-5}) is a \emph{linear} PDE, which has a unique solution
\begin{align} \label{reg-theta}
\theta^{(m)} \in C([0,T];H^1(\Omega))\;\; &\text{with}\;\; \partial_t \theta^{(m)}, \Delta_h \theta^{(m)}, \nabla_h \partial_z \theta^{(m)} \in L^2(\Omega \times (0,T)),  \\
&\text{and}   \;\;  \partial_t \partial_z \theta^{(m)}   \in    L^2(0,T; (H^1_h(\Omega))'),   \label{reg-theta1}
\end{align}
so that (\ref{Gl-5}) holds in the space $L^2(\Omega \times (0,T))$.
Then, we can put  ${\theta^{(m)}}$ back into (\ref{Gl-1}) to repeat the procedure to obtain $(\omega_{m+1}, \mathbf u_{m+1},  w_{m+1},  \theta^{(m+1)})^{tr}$. In conclusion, given $\theta^{(1)}$ satisfying (\ref{theta1}), by induction, the ``Galerkin-like" system (\ref{Gl-1})-(\ref{Gl-6}) has a unique solution $(\omega_m, \mathbf u_m,  w_m,  {\theta^{(m)}})^{tr}$ on $[0,T]$, for any $m\geq 2$.

We aim to show that the $H^1(\Omega)$ norm of $(\mathbf u_m,  w_m,  {\theta^{(m)}})^{tr}$ is bounded on $[0,T]$ uniformly in $m$, and there exists a subsequence converging to a solution $(\mathbf u,  w,  \theta)^{tr}$ of system (\ref{model-1})-(\ref{model-4}).

\begin{remark}
By assuming $\bar{\mathbf u_0}=0$, $\bar{w_0} =0$ and $\bar{\theta_0}=0$, we have
\begin{align} \label{mean0}
\bar{\mathbf u_m} =0 ,  \;\; \bar{w_m}=0 , \;\;  \bar{\theta^{(m)}}=0, \;\; \text{for all} \,\, t\in [0,T], \; m\geq 2.
\end{align}
Indeed, since (\ref{Gl-3})-(\ref{Gl-4}) hold, it is required that $\bar{\omega_m}=0$, $\bar{\mathbf u_m}=0$, and $\bar{\phi_m} =0$ for all $m\geq 2$.
To see that $\bar{w_m}=0$ and $\bar{\theta^{(m)}}=0$ for all $m\geq 2$, we use induction. First, note that $\bar{\theta^{(1)}}=0$, due to (\ref{theta1}) and $\bar{\theta_0}=0$. Now, we assume $\bar{\theta^{(m-1)}}=0$ for an $m\geq 2$, and show $\bar{w_m}=0$, $\bar{\theta^{(m)}}=0$. In fact, by taking the horizontal mean of each term of (\ref{Gl-1}) and using $\nabla_h \cdot \mathbf u_m=0$, we obtain $\partial_t \bar{w_m}=0$. Then, because $\bar{w_m}(0)=\bar{P_m w_0}=0$, it follows that $\bar{w_m}=0$. Next we take the horizontal mean of each term of the temperature equation (\ref{Gl-5}) to get $\partial_t \bar{\theta^{(m)}}=0$, which implies $\bar{\theta^{(m)}}=0$, since $\bar{\theta^{(m)}}(0)=\bar{\theta_0}=0$. Finally, to check whether  (\ref{mean0}) is consistent with equation (\ref{Gl-2}), we take the horizontal mean on (\ref{Gl-2}), then all terms vanish, if (\ref{mean0}) holds.
\end{remark}

\vspace{0.1 in}

\subsection{Uniform bound for $(\mathbf u_m,  w_m,  {\theta^{(m)}})^{tr}$ in $H^1(\Omega)$}     \label{sec-unib}
This section is devoted to proving that $(\mathbf u_m,  w_m,  {\theta^{(m)}})^{tr}$ has a uniform bound in $L^{\infty}(0,T; H^1(\Omega))$ independent of $m$.
It implies that $\omega_m$ is uniformly bounded in $L^{\infty}(0,T;L^2(\Omega))$.
The calculations in this section are legitimate, because solutions for (\ref{Gl-1})-(\ref{Gl-6}) are sufficiently regular. More precisely, $\omega_m$, $\mathbf u_m$ and $w_m$ are trigonometric polynomials satisfying (\ref{Gl-1})-(\ref{Gl-4}) in the classic sense, while $\theta^{(m)}$ has regularity (\ref{reg-theta})-(\ref{reg-theta1}) so that equation (\ref{Gl-5}) holds in $L^2(\Omega \times (0,T))$.

\subsubsection{Estimate for $\|{\theta^{(m)}}\|_2^2$}
For any $m \geq 2$, we multiply (\ref{Gl-3}) with ${\theta^{(m)}}$ and then integrate it over $\Omega \times [0,t]$ to get
\begin{align}  \label{theta-ene'}
\frac{1}{2}\|{\theta^{(m)}}(t)\|_2^2   + \int_0^t   \left(\frac{1}{Pe} \|\nabla_h {\theta^{(m)}}\|_2^2   +   \|\bar{w_m{\theta^{(m)}}}\|_2^2\right)  ds
=\frac{1}{2} \|{\theta^{(m)}}(0)\|_2^2 =  \frac{1}{2}\|\theta_0\|_2^2,
\end{align}
for all $t\in [0,T]$. Also, for $\theta^{(1)}$, we obtain from (\ref{theta1}) that
\begin{align}    \label{theta-ene0}
\frac{1}{2}\|\theta^{(1)}(t)\|_2^2   + \int_0^t  \frac{1}{Pe} \|\nabla_h \theta^{(1)}\|_2^2 ds  =      \frac{1}{2} \|\theta^{(1)}(0)\|_2^2
   = \frac{1}{2} \|\theta_0\|_2^2, \;\;  \text{for all} \;t\in [0,T].
\end{align}

\vspace{0.1 in}

\subsubsection{Estimate for $\|w_m\|_2^2 + \|\mathbf u_m\|_2^2$}
Taking the $L^2(\Omega)$ inner product of equations (\ref{Gl-1})-(\ref{Gl-2}) with $(w_m,\phi_m)^{tr}$ shows
\begin{align}    \label{L2-1}
&\frac{1}{2}\frac{d}{dt} \left(\|w_m\|_2^2 + \|\mathbf u_m\|_2^2 \right) +\frac{1}{Re} \left(\|\nabla_h w_m\|_2^2
+ \|\nabla_h \mathbf u_m\|_2^2 \right)  + \epsilon^2 \|\partial_z \phi_m\|_2^2    \notag\\
&= \Gamma ({\theta^{(m-1)}}, w_m)
\end{align}
where we have used identities (\ref{iden-1}), (\ref{iden-3}) and (\ref{iden-2}).

Since the horizontal mean $\bar{w_m}=0$ by (\ref{mean0}), one has the Poincar\'e inequality $\|w_m\|_2 \leq \gamma \|\nabla_h w_m\|_2$. Then
$\Gamma ({\theta^{(m-1)}}, w_m) \leq \Gamma \|{\theta^{(m-1)}}\|_2 \|w_m\|_2
\leq \frac{1}{2Re} \|\nabla_h w_m\|_2^2 + C \|{\theta^{(m-1)}}\|_2^2$. As a result,
\begin{align}    \label{L2-2}
\frac{d}{dt} \left(\|w_m\|_2^2 + \|\mathbf u_m\|_2^2 \right) +\frac{1}{Re} \left(\|\nabla_h w_m\|_2^2
+ \|\nabla_h \mathbf u_m\|_2^2 \right)  + \epsilon^2 \|\partial_z \phi_m\|_2^2  \leq C \|{\theta^{(m-1)}}\|_2^2.
\end{align}
Integrating (\ref{L2-2}) over $[0,t]$, we obtain, for $m\geq 2$,
\begin{align}  \label{L2-4}
&\|w_m(t)\|_2^2 + \|\mathbf u_m(t)\|_2^2 +\int_0^t \left(\frac{1}{Re} \left(\|\nabla_h w_m\|_2^2
+ \|\nabla_h \mathbf u_m\|_2^2 \right)  + \epsilon^2 \|\partial_z \phi_m\|_2^2 \right)ds  \notag\\
&\leq   \|w_0\|_2^2 + \|\mathbf u_0\|_2^2 + C\int_0^t  \|{\theta^{(m-1)}}\|_2^2 ds
\leq  \|w_0\|_2^2 + \|\mathbf u_0\|_2^2 + C \|\theta_0\|_2^2,
\end{align}
for all $t\in [0,T]$, where the last inequality is due to (\ref{theta-ene'}) and (\ref{theta-ene0}).
\vspace{0.1 in}

\subsubsection{Estimate for $\|\omega_m\|_2^2$} Taking the inner product of (\ref{Gl-2}) with $\omega_m$ yields
\begin{align}  \label{omega-1}
\frac{1}{2}\frac{d}{dt} \|\omega_m\|_2^2
+ \frac{1}{Re}\|\nabla_h \omega_m\|_2^2 + \epsilon^2 \|\partial_z \mathbf u_m\|_2^2 =   (\partial_z w_m, \omega_m),
\end{align}
where (\ref{iden-1}) and (\ref{iden-3}) have been used. Thanks to (\ref{iden-4}), after integration by parts, we have
\begin{align}   \label{omega-2}
&(\partial_z w_m,\omega_m)
=\int_{\Omega} \partial_z w_m (-\Delta_h \phi_m) dx dy dz = -\int_{\Omega} \nabla_{h}w_m \cdot  \nabla_h \partial_z \phi_m  dx dy dz  \notag\\
&\leq \|\nabla_h w_m\|_2\|\nabla_h \partial_z \phi_m\|_2 = \|\nabla_h w_m\|_2 \|\mathbf \partial_z \mathbf u_m\|_2
\leq \frac{\epsilon^2}{2} \|\partial_z \mathbf u_m\|_2^2 + \frac{1}{2\epsilon^2}  \|\nabla_h w_m\|_2^2.
\end{align}
Combining (\ref{omega-1}) and (\ref{omega-2}) implies
\begin{align}   \label{omega-3}
\frac{d}{dt} \|\omega_m\|_2^2
+ \frac{2}{Re}\|\nabla_h \omega_m\|_2^2 + \epsilon^2 \|\partial_z \mathbf u_m\|_2^2
\leq \frac{1}{\epsilon^2} \|\nabla_h w_m\|_2^2.
\end{align}
By integrating (\ref{omega-3}) over the interval $[0,t]$, we obtain, for $m\geq 2$,
\begin{align}  \label{omega}
&\|\omega_m(t)\|_2^2 + \int_0^t \left(\frac{2}{Re}\|\nabla_h \omega_m\|_2^2 + \epsilon^2 \|\partial_z \mathbf u_m\|_2^2  \right) ds
 \leq \|\omega_m(0)\|_2^2 + \frac{1}{\epsilon^2} \int_0^t  \|\nabla_h w_m\|_2^2  ds \notag\\
&\leq \|\omega_0\|_2^2 + C \left(\|w_0\|_2^2 + \|\mathbf u_0\|_2^2 + \|\theta_0\|_2^2\right),   \;\;\; \text{for all}\,\, t\in [0,T],
\end{align}
where the last inequality is due to (\ref{L2-4}).

\vspace{0.1 in}

\subsubsection{Estimate for $\|\nabla_h w_m\|_2^2$}
Taking the inner product of (\ref{Gl-1}) with $-\Delta_h w_m$ yields
\begin{align}    \label{H1w-00}
&\frac{1}{2} \frac{d}{dt}\|\nabla_h w_m\|_2^2 +\frac{1}{Re} \|\Delta_h w_m\|_2^2    \notag\\
&\leq \int_{\Omega} |(\mathbf u_m \cdot \nabla_h w_m) \Delta_h w_m|  dx dy dz
+ \|\partial_z \phi_m\|_2 \|\Delta_h w_m\|_2  + \Gamma \|{\theta^{(m-1)}}\|_2 \|\Delta_h w_m\|_2  \notag\\
&\leq C   \|\omega_m\|_2^{1/2} (\|\mathbf u_m\|_2 + \|\partial_z \mathbf u_m\|_2)^{1/2}   \|\nabla_h w_m\|_2^{1/2}  \|\Delta_h w_m\|_2^{3/2}  \notag\\
&\hspace{0.2 in} +\|\partial_z \phi_m\|_2 \|\Delta_h w_m\|_2  + \Gamma \|{\theta^{(m-1)}}\|_2 \|\Delta_h w_m\|_2,
\end{align}
where we have used Lemma \ref{lemma1} to establish the last inequality. Then, employing the Young's inequality, we obtain
\begin{align}     \label{H1w-0}
&\frac{d}{dt}\|\nabla_h w_m\|_2^2 +\frac{1}{Re} \|\Delta_h w_m\|_2^2           \notag\\
&\leq C\|\omega_m\|_2^2
(\|\mathbf u_m\|_2^2 + \|\partial_z \mathbf u_m\|_2^2) \|\nabla_h w_m\|_2^2 + C\left(\| \partial_z \phi_m\|_2^2
+ \|{\theta^{(m-1)}}\|_2^2\right).
\end{align}
Thanks to the Gronwall's inequality, we have, for $m\geq 2$,
\begin{align}  \label{H1w}
&\|\nabla_h w_m(t)\|_2^2 + \frac{1}{Re}\int_0^t \|\Delta_h w_m\|_2^2 ds \notag\\
&\leq  \left(\|\nabla_h w_m(0)\|_2^2 +C \int_0^t \left(\|\partial_z \phi_m\|_2^2 + \|{\theta^{(m-1)}}\|_2^2\right) ds\right)
e^{\int_0^t C \|\omega_m\|_2^2 (\|\mathbf u_m\|_2^2 + \|\partial_z \mathbf u_m\|_2^2) ds} \notag\\
&\leq C (\|\nabla_h w_0\|_2, \|\omega_0\|_2, \|\theta_0\|_2),   \;\;\;  \text{for all}\,\,  t\in [0,T],
\end{align}
where the last inequality is due to estimates (\ref{theta-ene'}), (\ref{theta-ene0}), (\ref{L2-4}) and (\ref{omega}).
\vspace{0.1 in}

\subsubsection{Estimate for $ \| \bar{|{\theta^{(m)}}|^2} \|_{L^{\infty}}$}

We multiply (\ref{Gl-3}) by $\theta^{(m)}$ and integrate the result with respect to horizontal variables over $[0,L]^2$.
Recall the notation for the horizontal mean   $\bar f=\frac{1}{L^2}\int_{[0,L]^2} f dx dy$.
Since $\nabla_h \cdot \mathbf u_m=0$, it follows that
\begin{align}  \label{theta-1}
\frac{1}{2}\frac{\partial}{\partial t} \bar{ |\theta^{(m)}|^2}(z) + \frac{1}{Pe} \bar{|\nabla_h \theta^{(m)}|^2}(z) +  \left(\bar{w_m \theta^{(m)}}\right)^2 (z) =0,
\end{align}
for a.e. $z\in [0,1]$. Integrating (\ref{theta-1}) over $[0,t]$ gives us
\begin{align}   \label{theta-inff}
&\frac{1}{2}\bar{|{\theta^{(m)}}|^2}(z,t)+\int_0^t  \left[\frac{1}{Pe} \bar{|\nabla_h {\theta^{(m)}}|^2}(z) +  \left(\bar{w_m{\theta^{(m)}}}\right)^2 (z) \right] ds  \notag\\
&=  \frac{1}{2}\bar{|{\theta^{(m)}}|^2}(z,0)   =   \frac{1}{2} \bar{\theta_0^2}(z)  \leq C\left( \|\theta_0\|_2^2 +  \|\partial_z \theta_0\|_2^2 \right),
\end{align}
for a.e. $z\in [0,1]$, for all $t\in [0,T]$, $m\geq 2$, where the last inequality is due to Lemma \ref{lemma2}.

\vspace{0.1 in}

\subsubsection{Estimate for $\|\nabla_h {\theta^{(m)}}\|_2^2$}
Recall the regularity of $\theta^{(m)}$ given in (\ref{reg-theta}).
Thus, we can take the $L^2(\Omega)$ inner product of (\ref{Gl-3}) with $-\Delta_h \theta^{(m)}$, and after integrating by parts, we obtain
\begin{align}     \label{H1ta-f1}
&\frac{1}{2} \frac{d}{dt} \|\nabla_h \theta^{(m)}\|_2^2 + \frac{1}{Pe} \|\Delta_h \theta^{(m)}\|_2^2 \notag\\
 & \leq \int_{\Omega} \left|(\mathbf u_m \cdot \nabla_h \theta^{(m)}) \Delta_h \theta^{(m)} \right| dx dy dz
 +  \int_0^1 \left|\bar{w_m \theta^{(m)}} \right| \left(\int_{[0,L]^2} |\theta^{(m)} \Delta_h w_m | dx dy \right) dz \notag\\
&\leq C \|\nabla_h \mathbf u_m\|_2^{1/2}   \left(\|\mathbf u_m\|_2+\| \partial_z  \mathbf u_m\|_2\right)^{1/2} \|\nabla_h \theta^{(m)}\|_2^{1/2} \|\Delta_h \theta^{(m)}\|_2^{3/2}  \notag\\
& \hspace{0.2 in} + \|\bar{w_m\theta^{(m)}}\|_2  \|\Delta_h w_m\|_2   \|\bar{|\theta^{(m)}|^2}\|_{L^{\infty}}^{1/2},
\end{align}
where we have used Lemma \ref{lemma1}, the Cauchy-Schwarz inequality, and Poincar\'e inequality (\ref{poin}).

Employing the Young's inequality implies
\begin{align} \label{H1ta-0}
\frac{d}{dt} \|\nabla_h \theta^{(m)}\|_2^2 + \frac{1}{Pe} \|\Delta_h \theta^{(m)}\|_2^2
\leq &C \|\nabla_h \mathbf u_m\|_2^2  \left(\|\mathbf u_m\|_2^2+\|  \partial_z \mathbf u_m\|_2^2 \right)
 \|\nabla_h \theta^{(m)}\|_2^2     \notag\\
 &+ \|\bar{w_m \theta^{(m)}} \|_2^2   \|\bar{|\theta^{(m)}|^2} \|_{L^{\infty}} +      \|\Delta_h w_m\|_2^2.
 \end{align}
Applying the Gronwall's inequality to (\ref{H1ta-0}), we have, for $m\geq 2$,
\begin{align}     \label{theta-nh0}
&\|\nabla_h {\theta^{(m)}}(t)\|_2^2    + \frac{1}{Pe} \int_0^t \|\Delta_h {\theta^{(m)}}\|_2^2 ds   \notag\\
&\leq \left[  \|\nabla_h \theta_0\|_2^2   +   \int_0^t  \left(   \|\bar{w_m  {\theta^{(m)}}}\|_2^2    \|\bar{|{\theta^{(m)}}|^2}\|_{L^{\infty}} + \|\Delta_h w_m\|_2^2 \right) ds   \right]
e^{C\int_0^t  \|\nabla_h \mathbf u_m\|_2^2    \left(\|\mathbf u_m\|_2^2+\| \partial_z \mathbf u_m\|_2^2 \right) ds}   \notag\\
&\leq  C( \|\theta_0\|_{H^1},   \|\nabla_h w_0\|_2, \|\omega_0\|_2   ),           \;\;\;  \text{for}\,\,  t\in [0,T],
\end{align}
due to the bounds (\ref{theta-ene'}), (\ref{omega}), (\ref{H1w}) and (\ref{theta-inff}).

\vspace{0.1 in}

\subsubsection{Estimate for $\|\partial_z w_m\|_2^2 + \| \partial_z \mathbf u_m\|_2^2 + \| \partial_z {\theta^{(m)}}\|_2^2 $}    \label{sec-z}
We differentiate (\ref{Gl-1})-(\ref{Gl-2}) with respect to $z$ and multiply them by $\partial_z w_m$ and $\partial_z \phi_m$ respectively.
Integrating the obtained equations over $\Omega \times [0,t]$ yields
\begin{align}   \label{H1z-1}
&\frac{1}{2}\left(\|\partial_z w_m(t)\|_2^2 + \| \partial_z \mathbf u_m(t)\|_2^2\right) +
\frac{1}{Re} \int_0^t \left(\|\nabla_h \partial_z  w_m\|_2^2 + \|\partial_z \omega_m\|_2^2 \right) ds
+ \epsilon^2  \int_0^t \| \partial_{zz} \phi_m\|_2^2 ds  \notag\\
&\leq  \frac{1}{2} \left(\|\partial_z  w_m(0)\|_2^2 + \|\partial_z \mathbf u_m(0)\|_2^2 \right)
+ \int_0^t \int_{\Omega} |(  \partial_z{\mathbf u_m} \cdot \nabla_h w_m) \partial_z w_m| dx dy dz ds \notag\\
& \hspace{0.1 in} +  \int_0^t \int_{\Omega} |(\mathbf u_m \cdot \nabla_h \partial_z \phi_m) \partial_z \omega_m| dx dy dz  ds
+ \frac{1}{2}\int_0^t \left(\| \partial_z  {\theta^{(m-1)}}\|_2^2 + \Gamma^2 \|\partial_z w_m\|_2^2\right) ds,
\end{align}
for $t\in [0,T]$, where we have used (\ref{iden-1}) and (\ref{iden-2}).

Now, we estimate each nonlinear term in (\ref{H1z-1}). By Lemma \ref{lemma1} with $f=\nabla_h w_m$, $g= \partial_z \mathbf u_m$
and $h= \partial_z w_m$, we have
\begin{align}   \label{H1z-2}
&\int_{\Omega} |(\partial_z \mathbf u_m \cdot \nabla_h w_m) \partial_z w_m| dx dy dz  \notag\\
&\leq C \|\Delta_h w_m\|_2^{1/2}    \left(\|\nabla_h w_m\|_2 + \|\nabla_h \partial_z w_m\|_2 \right)^{1/2}
 \| \partial_z  \mathbf u_m\|_2^{1/2}   \| \partial_z \omega_m\|_2  ^{1/2} \| \partial_z  w_m\|_2 \notag\\
&\leq C  \|\Delta_h w_m\|_2   \| \partial_z  \mathbf u_m\|_2^{1/2}   \| \partial_z \omega_m\|_2  ^{1/2} \| \partial_z  w_m\|_2  \notag\\
&\hspace{0.2 in}+   C\|\Delta_h w_m\|_2^{1/2}   \|\nabla_h \partial_z w_m\|_2^{1/2}    \| \partial_z  \mathbf u_m\|_2^{1/2}   \| \partial_z \omega_m\|_2  ^{1/2} \| \partial_z  w_m\|_2\notag\\
&\leq \frac{1}{4Re}\left(\|\nabla_h \partial_z w_m\|_2^2 + \| \partial_z \omega_m\|_2^2 \right)
+   \| \partial_z \mathbf u_m\|_2^2 +  C \left(\|\Delta_h w_m\|_2^2
+   \| \partial_z \mathbf u_m\|_2^2\right) \| \partial_z w_m\|_2^2.
\end{align}
Also using Lemma \ref{lemma1} with $f=\mathbf u_m$, $g=\nabla_h \partial_z \phi_m$ and $h=\partial_z \omega_m$, we obtain
\begin{align} \label{H1z-4}
&\int_{\Omega} |(\mathbf u_m \cdot \nabla_h  \partial_z  \phi_m) \partial_z \omega_m| dx dy dz  \notag\\
&\leq C \|\omega_m\|_2^{1/2} \left(\|\mathbf u_m\|_2 + \| \partial_z  \mathbf u_m\|_2\right)^{1/2}
\| \partial_z  \mathbf u_m\|_2^{1/2}  \| \partial_z \omega_m\|_2^{3/2} \notag\\
&\leq \frac{1}{4Re}\|  \partial_z  \omega_m\|_2^2 + C \|\omega_m\|_2^2  \left(\|\mathbf u_m\|_2^2
+ \| \partial_z  \mathbf u_m\|_2^2 \right)   \| \partial_z   \mathbf u_m\|_2^2.
\end{align}

Applying (\ref{H1z-2})-(\ref{H1z-4}) to the right-hand side of inequality (\ref{H1z-1}) yields
\begin{align}      \label{H1z-uw}
&\|\partial_z w_m(t)\|_2^2 + \| \partial_z \mathbf u_m(t)\|_2^2 +
\frac{1}{Re} \int_0^t \left(\|\nabla_h \partial_z  w_m\|_2^2 + \|\partial_z \omega_m\|_2^2 \right) ds
+ \epsilon^2  \int_0^t \| \partial_{zz} \phi_m\|_2^2 ds  \notag\\
&\leq  \|\partial_z  w_m(0)\|_2^2 + \|\partial_z \mathbf u_m(0)\|_2^2
  + C\int_0^t    \left(   \|\omega_m\|_2^2 \|\mathbf u_m\|_2^2
+   \|\omega_m\|_2^2 \| \partial_z  \mathbf u_m\|_2^2 +1 \right) \| \partial_z \mathbf u_m\|_2^2 ds  \notag\\
&\hspace{0.2 in} + C  \int_0^t  \left(\|\Delta_h w_m\|_2^2  +   \| \partial_z \mathbf u_m\|_2^2 +1 \right) \|\partial_z w_m\|_2^2 ds + \int_0^t \| \partial_z  {\theta^{(m-1)}}\|_2^2  ds,
\end{align}
for all $t\in [0,T]$.

Next, we estimate $\|\partial_z \theta^{(m)}\|_2^2$.
Since $\theta^{(m)}$ has regularity (\ref{reg-theta})-(\ref{reg-theta1}) and $\mathbf u_m$, $w_m$ are analytic, we can differentiate (\ref{Gl-3}) with respect to $z$, and then multiply it by $\partial_z \theta^{(m)}$, and finally integrate the result with respect to horizontal variables over $[0,L]^2$ to obtain
\begin{align*}
&\frac{1}{2} \frac{d}{dt} \int_{[0,L]^2} | \partial_z \theta^{(m)}|^2 dx dy + \int_{[0,L]^2} ({ \partial_z \mathbf u_m} \cdot \nabla_h \theta^{(m)})  \partial_z \theta^{(m)} dx dy \notag\\
&\hspace{0.2 in}+ \bar{w_m\theta^{(m)}} \int_{[0,L]^2} (\partial_z w_m) (\partial_z \theta^{(m)} )dx dy    + \bar{(\partial_z w_m) \theta^{(m)}} \int_{[0,L]^2} w_m (\partial_z \theta^{(m)}) dx dy   \notag\\
& \hspace{0.2 in}  + \bar{w_m (\partial_z \theta^{(m)})} \int_{[0,L]^2} w_m (\partial_z \theta^{(m)}) dx dy \notag\\
&= - \frac{1}{Pe} \int_{[0,L]^2} |\nabla_h   \partial_z \theta^{(m)} |^2 dx dy,     \;\; \;   \text{for a.e.}   \,\, z\in [0,1].
\end{align*}
Recall the notation for the horizontal mean $\bar f =  \frac{1}{L^2}\int_{[0,L]^2} f dx dy.$  Therefore,
\begin{align}   \label{nH1z-2}
&\frac{1}{2} \frac{d}{dt} \int_{[0,L]^2}    |\partial_z \theta^{(m)}|^2 dx dy +   \frac{1}{Pe} \int_{[0,L]^2} |\nabla_h  \partial_z \theta^{(m)}|^2 dx dy +
L^2 \left|\bar{w_m (\partial_z \theta^{(m)})}\right|^2   \notag\\
&\leq  \int_{[0,L]^2}   \left| ({   \partial_z  \mathbf u_m} \cdot \nabla_h \theta^{(m)})  \partial_z \theta^{(m)}  \right| dx dy +
\left|\bar{w_m \theta^{(m)}}\right| \int_{[0,L]^2} |\partial_z w_m|   |\partial_z \theta^{(m)} | dx dy    \notag\\
&\hspace{0.2 in} + L^2  \left|\bar{ (\partial_z  w_m)     \theta^{(m)}} \right|  \left|\bar{w_m (\partial_z \theta^{(m)})}\right|,
 \;\; \;   \text{for a.e.}   \,\, z\in [0,1].
 \end{align}
Note,
\begin{align*}
&L^2 \left(\left|\bar{  (\partial_z w_m)   \theta^{(m)}}\right| \left|\bar{w_m   (\partial_z \theta^{(m)})} \right|\right)(z)
\leq   L^2 \left|\bar{  (\partial_z w_m)   \theta^{(m)}}\right|^2 (z) + L^2 \left|\bar{w_m   (\partial_z \theta^{(m)})}\right|^2(z)  \notag\\
&\leq   \bar{|\theta^{(m)}|^2}(z) \int_{[0,L]^2}   |\partial_z w_m (x,y,z) |^2   dx dy + L^2 \left|\bar{w_m (\partial_z \theta^{(m)})} \right|^2(z), \;\; \;   \text{for a.e.}   \,\, z\in [0,1].
\end{align*}
Thus, (\ref{nH1z-2}) implies
\begin{align}   \label{nH1z-4}
&\frac{1}{2} \frac{d}{dt} \int_{[0,L]^2}  |\partial_z \theta^{(m)} (x,y,z)|^2  dx dy +   \frac{1}{Pe} \int_{[0,L]^2} |\nabla_h \partial_z  \theta^{(m)} (x,y,z)|^2 dx dy  \notag\\
&\leq \int_{[0,L]^2}   \left| ({\partial_z  \mathbf u_m} \cdot \nabla_h \theta^{(m)})  \partial_z \theta^{(m)}  \right| (x,y,z) dx dy
+  \bar{|\theta^{(m)}|^2}(z)  \int_{[0,L]^2} |\partial_z w_m (x,y,z)|^2 dx dy
  \notag\\
& \hspace{0.2 in} + \left|\bar{w_m \theta^{(m)}} \right| (z)  \left(\int_{[0,L]^2} |\partial_z w_m (x,y,z) |^2  dx dy\right)^{1/2}
\left(\int_{[0,L]^2}  |\partial_z \theta^{(m)} (x,y,z) |^2  dx dy\right)^{1/2},
\end{align}
for a.e. $z\in [0,1]$. Applying Lemma \ref{lemma-3} to (\ref{nH1z-4}) with
$\eta (z,t)=  \left(\int_{[0,L]^2}  |\partial_z \theta^{(m)} (x,y,z,t)|^2  dx dy\right)^{1/2}$, and using Cauchy-Schwarz inequality, we deduce
\begin{align}   \label{nH1z-5}
&\int_{[0,L]^2}    |\partial_z \theta^{(m)}(x,y,z,t)|^2 dx dy  +  \frac{2}{Pe} \int_0^t    \int_{[0,L]^2} |\nabla_h \partial_z   \theta^{(m)}  (x,y,z,s)|^2 dx dy ds   \notag\\
&\leq C\int_{[0,L]^2}  | \partial_z \theta^{(m)}(x,y,z,0)|^2 dx dy  + C\int_0^t    \int_{[0,L]^2}
\left| ({  \partial_z \mathbf u_m} \cdot \nabla_h \theta^{(m)})  \partial_z \theta^{(m)}  \right| (x,y,z,s) dx dy  ds    \notag\\
&  \hspace{0.3 in} +  C \int_0^t  \left(  \bar{|\theta^{(m)}|^2} (z,s)   \int_{[0,L]^2} |\partial_z w_m (x,y,z,s)|^2 dx dy\right) ds     \notag\\
&  \hspace{0.3 in}  +  C \left(\int_0^t  \left|\bar{w_m\theta^{(m)}}\right|^2 (z,s) ds\right)   \int_0^t  \int_{[0,L]^2} |\partial_z w_m (x,y,z,s)|^2 dx dy  ds ,
\end{align}
for a.e. $z\in [0,1]$, and for all $t\in [0,T]$.

Then, we integrate (\ref{nH1z-5}) with respect to $z$ over $[0,1]$ to get
\begin{align}  \label{nH1z-6}
&\|  \partial_z \theta^{(m)}(t)\|^2_2 + \frac{2}{Pe} \int_0^t  \|\nabla_h \partial_z \theta^{(m)}\|_2^2 ds     \notag\\
&\leq  C\|\partial_z \theta^m(0)\|^2_2 + C \int_0^t    \int_{\Omega}  \left| ({ \partial_z   \mathbf u_m} \cdot \nabla_h \theta^{(m)} )  \partial_z \theta^{(m)}  \right| dx dy dz ds    \notag\\
& +  C      \int_0^t   \|\bar{|\theta^{(m)}|^2}\|_{L^{\infty}}      \|\partial_z w_m\|_2^2  ds
  +C  \left(\sup_{z\in [0,1]}  \int_0^t  \left|\bar{w_m\theta^{(m)}} \right|^2 (z,s) ds  \right)  \int_0^t     \| \partial_z w_m\|_2^2  ds.
\end{align}
Note that by using Lemma \ref{lemma1}, with $f=\nabla_h \theta^{(m)}$, $g= \partial_z \mathbf u_m$ and $h=  \partial_z \theta^{(m)}$, and Poincar\'e inequality (\ref{poin}), one has
\begin{align}   \label{H1z-5}
&C\int_{\Omega} \left|(   \partial_z \mathbf u_m \cdot \nabla_h \theta^{(m)}) \partial_z \theta^{(m)} \right| dx dy dz \notag\\
&\leq C \|\Delta_h \theta^{(m)}\|_2^{1/2}    \left(\|\nabla_h \theta^{(m)}\|_2 + \|\nabla_h \partial_z \theta^{(m)}\|_2\right)^{1/2}
\|  \partial_z \mathbf u_m\|_2^{1/2} \|\partial_z \omega_m\|_2^{1/2}  \| \partial_z \theta^{(m)}\|_2 \notag\\
&\leq \frac{1}{Pe}\|\nabla_h \partial_z \theta^{(m)}\|_2^2 + \frac{1}{2Re} \|\partial_z \omega_m\|_2^2
+ \|   \partial_z \mathbf u_m\|_2^2+ C\left(\|\Delta_h \theta^{(m)}\|_2^2 + \| \partial_z  \mathbf u_m\|_2^2   \right) \|  \partial_z \theta^{(m)}\|_2^2.
\end{align}
Substituting (\ref{H1z-5}) into (\ref{nH1z-6}), we obtain
\begin{align}    \label{H1z-t}
&\| \partial_z {\theta^{(m)}}(t)\|_2^2 + \frac{1}{Pe} \int_0^t  \|\nabla_h \partial_z {\theta^{(m)}}\|_2^2 ds   \notag\\
&\leq C \|\partial_z {\theta^{(m)}}(0)\|_2^2  +   \frac{1}{2 Re} \int_0^t   \|\partial_z \omega_m\|_2^2 ds  + \int_0^t  \|  \partial_z \mathbf u_m\|_2^2 ds       \notag\\
&  \hspace{0.2 in}
+   C \int_0^t  \left( \|\Delta_h {\theta^{(m)}}\|_2^2  + \| \partial_z \mathbf u_m\|_2^2 \right)  \|  \partial_z  {\theta^{(m)}}\|_2^2 ds
 +       C      \int_0^t   \|\bar{  |{\theta^{(m)}}|^2}\|_{L^{\infty}}      \|\partial_z w_m\|_2^2  ds  \notag\\
&  \hspace{0.2 in}
+   C  \left(\sup_{z\in [0,1]}  \int_0^t  \left|\bar{w_m {\theta^{(m)}}}\right|^2 (z,s) ds  \right)  \int_0^t     \|\partial_z w_m\|_2^2  ds,
\;\;\;   \text{for all} \,\,   t\in [0,T].
\end{align}

Combining (\ref{H1z-uw}) and (\ref{H1z-t}) provides
\begin{align}  \label{H1z-uwt}
&\|\partial_z w_m(t)\|_2^2 + \| \partial_z \mathbf u_m(t)\|_2^2 +       \| \partial_z {\theta^{(m)}}(t)\|_2^2 + \frac{1}{2Re} \int_0^t \left(\|\nabla_h \partial_z  w_m\|_2^2 + \|\partial_z \omega_m\|_2^2 \right) ds  \notag\\
& \hspace{0.2 in} + \frac{1}{Pe} \int_0^t  \|\nabla_h \partial_z {\theta^{(m)}}\|_2^2 ds   + \epsilon^2  \int_0^t \| \partial_{zz} \phi_m\|_2^2 ds    \notag\\
 &\leq     \|\partial_z  w_m(0)\|_2^2 + \|\partial_z \mathbf u_m(0)\|_2^2   +     C \|\partial_z {\theta^{(m)}}(0)\|_2^2    + C  \int_0^t  \left(\|\Delta_h w_m\|_2^2
+   \| \partial_z \mathbf u_m\|_2^2 +1 \right) \|\partial_z w_m\|_2^2 ds     \notag\\
& \hspace{0.2 in}    +          C      \int_0^t   \|\bar{   |{\theta^{(m)}}|^2}\|_{L^{\infty}}      \|\partial_z w_m\|_2^2  ds
+ C  \left(   \sup_{z\in [0,1]}  \int_0^t  \left|\bar{w_m {\theta^{(m)}}}\right|^2 ds  \right)  \int_0^t     \|\partial_z w_m\|_2^2  ds
\notag\\
& \hspace{0.2 in}+C \int_0^t \left(\|\omega_m\|_2^2  \|\mathbf u_m\|_2^2
+  \|\omega_m\|_2^2  \| \partial_z  \mathbf u_m\|_2^2 + 1\right)   \| \partial_z   \mathbf u_m\|_2^2 ds  \notag\\
& \hspace{0.2 in} +C \int_0^t  \left( \|\Delta_h {\theta^{(m)}}\|_2^2  + \| \partial_z \mathbf u_m\|_2^2 \right)  \|  \partial_z  {\theta^{(m)}}\|_2^2 ds
+ \int_0^t  \| \partial_z  {\theta^{(m-1)}}\|_2^2  ds,
\end{align}
for all $t\in [0,T]$.

Thanks to the Gronwall's inequality, we obtain, for all $t\in [0,T]$, $m\geq 2$,
\begin{align}   \label{H1z-11}
&\|\partial_z w_m(t)\|_2^2 + \| \partial_z \mathbf u_m(t)\|_2^2 +       \| \partial_z {\theta^{(m)}}(t)\|_2^2 + \frac{1}{2Re} \int_0^t \left(\|\nabla_h \partial_z  w_m\|_2^2 + \|\partial_z \omega_m\|_2^2 \right) ds  \notag\\
& \hspace{0.1 in} + \frac{1}{Pe} \int_0^t  \|\nabla_h \partial_z {\theta^{(m)}}\|_2^2 ds   + \epsilon^2  \int_0^t \| \partial_{zz} \phi_m\|_2^2 ds    \notag\\
&\leq \left( \|\partial_z w_m(0)\|_2^2 + \|\partial_z \mathbf u_m(0)\|_2^2 + C\|\partial_z {\theta^{(m)}}(0)\|_2^2   +
\int_0^t    \| \partial_z  {\theta^{(m-1)}}\|_2^2 ds  \right) e^{M(t)},
\end{align}
where
\begin{align}    \label{tildeC}
M(t)&=C\int_0^t  \Big(\|\Delta_h w_m\|_2^2 + \|   \partial_z \mathbf u_m\|_2^2
+ \|\omega_m\|_2^2       \|\mathbf u_m\|_2^2      + \|\omega_m\|_2^2  \|\partial_z \mathbf u_m\|_2^2    \notag\\
&   \hspace{0.5 in}  +  \|\Delta_h {\theta^{(m)}}\|_2^2   +   \|\bar{ | {\theta^{(m)}}|^2} \|_{L^{\infty}}
 + \sup_{z\in [0,1]}  \int_0^t  \left|\bar{w_m {\theta^{(m)}}}\right|^2 d\tau      +1  \Big) ds     \notag\\
&\leq \tilde C(\|\mathbf u_0\|_{H^1}, \|w_0\|_{H^1},   \|\theta_0\|_{H^1} , T),      \;\;\; \text{for all} \,\, t\in [0,T],
\end{align}
owing to estimates (\ref{L2-4}), (\ref{omega}), (\ref{H1w}), (\ref{theta-inff}) and (\ref{theta-nh0}).

In order to show the right-hand side of (\ref{H1z-11}) is uniformly bounded, we must prove that $\int_0^T    \| \partial_z  {\theta^{(m)}}\|_2^2 ds$ is uniformly bounded.
Set  $t_0 = 1/(2 e^{\tilde C}) $, where $\tilde C>0$ is given in (\ref{tildeC}). We first show $\int_0^{t_0}   \| \partial_z  {\theta^{(m)}}\|_2^2 ds  $ is a bounded sequence.
Then we show $\int_{t_0}^{2t_0}   \| \partial_z  {\theta^{(m)}}\|_2^2 ds $ is a bounded sequence. After iterating finitely many times, we will conclude that the sequence $\int_0^T \| \partial_z  {\theta^{(m)}}\|_2^2 ds$ is bounded.

To see $\int_0^{t_0}   \| \partial_z  {\theta^{(m)}}\|_2^2 ds$ is a bounded sequence, we argue by induction. To begin with,
we obtain from (\ref{theta1}) that
\begin{align}   \label{theta1z}
\int_0^t \|\partial_z \theta^{(1)}\|_2^2 ds \leq C\|\partial_z \theta^{(1)}(0)\|_2^2 =  C \|\partial_z \theta_0\|_2^2,   \;\;\; \text{for all}\,\, t\in [0,T].
\end{align}
For an index $m\geq 2$, we consider two cases.
If  $\int_0^{t_0}   \| \partial_z  {\theta^{(m)}}\|_2^2 ds <   \int_0^{t_0}    \| \partial_z  {\theta^{(m-1)}}\|_2^2 ds $,
then the sequence $\int_0^{t_0} \| \partial_z  {\theta^{(m)}}\|_2^2 ds$ decreases at the level $m$.
Conversely, if $\int_0^{t_0} \| \partial_z  {\theta^{(m)}}\|_2^2 ds \geq  \int_0^{t_0} \| \partial_z  {\theta^{(m-1)}}\|_2^2 ds $, then by (\ref{H1z-11})-(\ref{tildeC}), we obtain, for $t\in [0,t_0]$,
\begin{align} \label{t0}
&\|\partial_z w_m(t)\|_2^2 + \| \partial_z \mathbf u_m(t)\|_2^2 + \|\partial_z {\theta^{(m)}}(t)\|_2^2   \notag\\
&\leq  \left( \|\partial_z w_m(0)\|_2^2 + \|\partial_z \mathbf u_m(0)\|_2^2 + C\|\partial_z {\theta^{(m)}}(0)\|_2^2   +
\int_0^{t_0}   \| \partial_z  {\theta^{(m-1)}}\|_2^2 ds  \right) e^{M(t)}  \notag\\
&\leq   e^{\tilde C}   \left( \|\partial_z w_0\|_2^2 + \|\partial_z \mathbf u_0\|_2^2 + C\|\partial_z \theta_0\|_2^2  \right)   +     e^{\tilde C}\int_0^{t_0}   \| \partial_z  \theta^{(m)}\|_2^2 ds      \notag\\
&\leq  C(\|\mathbf u_0\|_{H^1}, \|w_0\|_{H^1},   \|\theta_0\|_{H^1} , T)  +e^{\tilde C}\int_0^{t_0}   \| \partial_z  \theta^{(m)}\|_2^2 ds,
\end{align}
where $\tilde C>0$ is the constant given in (\ref{tildeC}), depending only on $\|\mathbf u_0\|_{H^1}$, $\|w_0\|_{H^1}$, $\|\theta_0\|_{H^1}$ and $T$.

Integrating (\ref{t0}) over $[0,t_0]$ provides
\begin{align}    \label{t0-1}
\int_0^{t_0}   \| \partial_z {\theta^{(m)}}\|_2^2 ds \leq    C(\|\mathbf u_0\|_{H^1}, \|w_0\|_{H^1},   \|\theta_0\|_{H^1} , T)  +  t_0 \cdot e^{\tilde C}\int_0^{t_0}   \| \partial_z  \theta^{(m)}\|_2^2 ds.
\end{align}
Since $t_0 = 1/(2 e^{\tilde C})$, then $t_0 \cdot e^{\tilde C}=\frac{1}{2}$. Thus, we obtain from (\ref{t0-1}) that
\begin{align}    \label{t0-2}
\int_0^{t_0}   \| \partial_z {\theta^{(m)}}\|_2^2 ds  \leq    C(\|\mathbf u_0\|_{H^1}, \|w_0\|_{H^1},   \|\theta_0\|_{H^1} , T).
\end{align}
Since the right-hand side of (\ref{t0-2}) is independent of $m$, by induction, we obtain the sequence $\int_0^{t_0}   \| \partial_z {\theta^{(m)}}\|_2^2 ds$ is bounded by
$C(\|\mathbf u_0\|_{H^1}, \|w_0\|_{H^1},   \|\theta_0\|_{H^1} , T)$. As a result, by (\ref{H1z-11}), we have
\begin{align}    \label{H1z-122}
&\|\partial_z w_m(t)\|_2^2 + \| \partial_z \mathbf u_m(t)\|_2^2 + \|\partial_z {\theta^{(m)}}(t)\|_2^2   \notag\\
&\leq C(\|\mathbf u_0\|_{H^1}, \|w_0\|_{H^1},   \|\theta_0\|_{H^1} , T),    \;\;  \text{for all}\,\,  t\in [0,t_0], \,m\geq 2.
\end{align}

We remark that the constant $C(\|\mathbf u_0\|_{H^1}, \|w_0\|_{H^1},   \|\theta_0\|_{H^1} , T)$ may vary from line to line in our estimates, but it is always independent of $m$. On the other hand, $\tilde C$ is a fixed constant given in (\ref{tildeC}), also independent of $m$.

Next, we show that $\int_{t_0}^{2t_0} \|\partial_z {\theta^{(m)}}\|_2^2 ds$ is a bounded sequence.
Indeed, repeating the same estimates as in (\ref{H1z-1})-(\ref{tildeC}) with the time starting at $t_0$, we have, for all $t\in [t_0,T]$,
\begin{align}   \label{t1}
&\|\partial_z w_m(t)\|_2^2 + \| \partial_z \mathbf u_m(t)\|_2^2 +       \| \partial_z {\theta^{(m)}}(t)\|_2^2    \notag\\
&\leq \left( \|\partial_z w_m(t_0)\|_2^2 + \|\partial_z \mathbf u_m(t_0)\|_2^2 + C\|\partial_z {\theta^{(m)}}(t_0)\|_2^2   +
\int_{t_0}^t    \| \partial_z  {\theta^{(m-1)}}\|_2^2 ds  \right) e^{M_1(t)},
\end{align}
for $m\geq 2$, where
\begin{align}     \label{t2}
M_1(t)&=C\int_{t_0}^t  \big(\|\Delta_h w_m\|_2^2 + \|   \partial_z \mathbf u_m\|_2^2
+ \|\omega_m\|_2^2       \|\mathbf u_m\|_2^2      + \|\omega_m\|_2^2  \|\partial_z \mathbf u_m\|_2^2    \notag\\
&   \hspace{0.5 in}  +  \|\Delta_h {\theta^{(m)}}\|_2^2   +   \|\bar{    |{\theta^{(m)}}|^2} \|_{L^{\infty}}
 + \sup_{z\in [0,1]}  \int_{t_0}^t  \left|\bar{w_m {\theta^{(m)}}} \right|^2 d\tau      +1  \big) ds     \notag\\
&\leq M(t)\leq \tilde C(\|\mathbf u_0\|_{H^1}, \|w_0\|_{H^1},   \|\theta_0\|_{H^1} , T),      \;\;\; \text{for all} \,\, t\in [t_0,T],
\end{align}
due to (\ref{tildeC}).

Now, we can see that $\int_{t_0}^{2t_0} \|\partial_z {\theta^{(m)}}\|_2^2 ds$ is a bounded sequence via induction. Indeed, we should first notice that $\int_{t_0}^{2t_0} \|\partial_z \theta^{(1)}\|_2^2 ds \leq C\|\partial_z \theta_0\|_2^2$ due to (\ref{theta1z}). Then, for an index $m\geq 2$, if $\int_{t_0}^{2t_0} \| \partial_z  {\theta^{(m)}}\|_2^2 ds \geq  \int_{t_0}^{2t_0} \| \partial_z  {\theta^{(m-1)}}\|_2^2 ds$,
 then by (\ref{t1}), we obtain, for $t\in [t_0, 2t_0]$,
 \begin{align}   \label{t3}
&\|\partial_z w_m(t)\|_2^2 + \| \partial_z \mathbf u_m(t)\|_2^2 +       \| \partial_z {\theta^{(m)}}(t)\|_2^2    \notag\\
&\leq \left( \|\partial_z w_m(t_0)\|_2^2 + \|\partial_z \mathbf u_m(t_0)\|_2^2 + C\|\partial_z {\theta^{(m)}}(t_0)\|_2^2   +
\int_{t_0}^{2t_0}   \| \partial_z  {\theta^{(m-1)}}\|_2^2 ds  \right) e^{M_1(t)},  \notag\\
&\leq    \left( C(\|\mathbf u_0\|_{H^1}, \|w_0\|_{H^1},   \|\theta_0\|_{H^1} , T)   +
\int_{t_0}^{2t_0}   \| \partial_z  \theta^{(m)}\|_2^2 ds  \right) e^{\tilde C}   \notag\\
&\leq   C(\|\mathbf u_0\|_{H^1}, \|w_0\|_{H^1},   \|\theta_0\|_{H^1} , T) +  e^{\tilde C}   \int_{t_0}^{2t_0}   \| \partial_z  \theta^{(m)}\|_2^2 ds,
\end{align}
where we have used (\ref{H1z-122}) and (\ref{t2}). Then integrating (\ref{t3}) over $[t_0, 2t_0]$ and using $t_0 \cdot e^{\tilde C}=\frac{1}{2}$, we have
\begin{align}    \label{t4}
\int_{t_0}^{2t_0}   \| \partial_z {\theta^{(m)}}\|_2^2 ds  \leq    C(\|\mathbf u_0\|_{H^1}, \|w_0\|_{H^1},   \|\theta_0\|_{H^1} , T).
\end{align}
Hence, by induction, we see that $\int_{t_0}^{2t_0}   \| \partial_z {\theta^{(m)}}\|_2^2 ds$ is bounded by $C(\|\mathbf u_0\|_{H^1}, \|w_0\|_{H^1},   \|\theta_0\|_{H^1}, T)$ for all $m \geq 2$. Thus, by (\ref{t1}), we obtain, for all $t\in [t_0, 2t_0]$,
\begin{align*}
\|\partial_z w_m(t)\|_2^2 + \| \partial_z \mathbf u_m(t)\|_2^2 + \| \partial_z {\theta^{(m)}}(t)\|_2^2 \leq   C(\|\mathbf u_0\|_{H^1}, \|w_0\|_{H^1},   \|\theta_0\|_{H^1} , T).\end{align*}

After iterating the above procedure on finitely many intervals $[0,t_0]$, $[t_0,2t_0]$, $\cdots$, $[nt_0, T]$, we eventually obtain $\int_0^T   \| \partial_z {\theta^{(m)}}\|_2^2 dt$ is a bounded sequence, namely,
\begin{align}    \label{t5}
\int_0^T   \| \partial_z {\theta^{(m)}}\|_2^2 dt  \leq   C(\|\mathbf u_0\|_{H^1}, \|w_0\|_{H^1},   \|\theta_0\|_{H^1} , T), \;\;\;  \text{for any} \,\, m\in \mathbb N.
\end{align}
By substituting (\ref{t5}) to the right-hand side of (\ref{H1z-11}), we achieve the desired uniform bound
\begin{align}   \label{bdz}
&\|\partial_z w_m(t)\|_2^2 + \| \partial_z \mathbf u_m(t)\|_2^2 +       \| \partial_z {\theta^{(m)}}(t)\|_2^2 + \frac{1}{2Re} \int_0^t \left(\|\nabla_h \partial_z  w_m\|_2^2 + \|\partial_z \omega_m\|_2^2 \right) ds  \notag\\
& \hspace{0.1 in} + \frac{1}{Pe} \int_0^t  \|\nabla_h \partial_z {\theta^{(m)}}\|_2^2 ds   + \epsilon^2  \int_0^t \| \partial_{zz} \phi_m\|_2^2 ds    \notag\\
&\leq C(\|\mathbf u_0\|_{H^1}, \|w_0\|_{H^1},   \|\theta_0\|_{H^1} , T),    \;\;\; \text{for all}\,\, t\in [0,T],  \, m\geq 2.
\end{align}

\vspace{0.1 in}
\subsection{Passage to the limit}
According to all of the estimates which have been established in section \ref{sec-unib} for $\mathbf u_m$, $\omega_m$, $w_m$ and ${\theta^{(m)}}$, we obtain the following uniform bounds:
\begin{align}
&\mathbf u_m,\; w_m, \;{\theta^{(m)}}  \;\;\text{are uniformly bounded in}\;\;  L^{\infty}(0,T; H^1(\Omega));   \label{Gal-2} \\
&\omega_m       \;\;\text{is uniformly bounded in}\;\;  L^{\infty}(0,T; L^2(\Omega));      \label{Gal-1}    \\
&\nabla_h \omega_m, \; \Delta_h w_m, \; \Delta_h {\theta^{(m)}}, \; \bar{w_m {\theta^{(m)}}}\;\;\text{are uniformly bounded in}\;\; L^2(\Omega \times (0,T)); \label{Gal-3} \\
&   \partial_z \omega_m, \; \nabla_h \partial_z w_m, \; \nabla_h \partial_z {\theta^{(m)}}, \;  \partial_{zz}\phi_m \;\;\text{are uniformly bounded in}\;\; L^2(\Omega \times (0,T)). \label{Gal-4}
\end{align}

Therefore, on a subsequence, as $m\rightarrow \infty$,
\begin{align}
&\mathbf u_m \rightarrow \mathbf u, \; w_m \rightarrow w, \; {\theta^{(m)}} \rightarrow \theta  \;\; \text{weakly$^*$ in} \;\; L^{\infty}(0,T;H^1(\Omega));   \label{Gal-5'}  \\
& \omega_m \rightarrow \omega         \;\; \text{weakly$^*$ in} \;\; L^{\infty}(0,T; L^2(\Omega));    \\
&\nabla_h \omega_m \rightarrow \nabla_h \omega,       \;
\Delta_h w_m \rightarrow \Delta_h w, \; \Delta_h {\theta^{(m)}} \rightarrow \Delta_h \theta \;\; \text{weakly in}  \;\; L^2(\Omega \times (0,T));    \label{Gal-5''}  \\
&  \partial_z \omega_m \rightarrow \omega_z,       \;
\nabla_h \partial_z w_m \rightarrow \nabla_h w_z, \; \nabla_h \partial_z {\theta^{(m)}} \rightarrow \nabla_h \theta_z,  \; \partial_{zz}\phi_m \rightarrow \phi_{zz} \;\; \text{weakly in}  \;\; L^2(\Omega \times (0,T)).     \label{Gal-5'''}
\end{align}

In order to use a compactness theorem to obtain certain strong convergence of the approximate solutions, we shall derive uniform bounds independent of $m\geq 2$
for $\partial_t w_m$, $\partial_t \mathbf u_m$, $\partial_t \omega_m$ and $\partial_t {\theta^{(m)}}$. First, we claim that the sequence $\partial_t w_m$ is uniformly bounded in $L^2(\Omega \times (0,T))$. Indeed, for any function $\eta\in L^{4/3}(0,T;L^2(\Omega))$, we use Lemma \ref{lemma1} to estimate
\begin{align}  \label{Gal-5}
&\int_0^T \int_{\Omega} (\mathbf u_m \cdot \nabla_h w_m) \eta  dx dy dz  dt    \notag\\
&\leq C\int_0^T \|\omega_m\|_2^{1/2} \left(\|\mathbf u_m\|_2+ \|\partial_z \mathbf u_m\|_2\right)^{1/2} \|\nabla_h w_m\|_2^{1/2}
\|\Delta_h w_m\|_2^{1/2} \|\eta\|_2 dt  \notag\\
&\leq C\sup_{t\in [0,T]} \Big[\|\omega_m\|_2^{\frac{1}{2}} (\|\mathbf u_m\|_2^{\frac{1}{2}}+ \|\partial_z \mathbf u_m\|_2^{\frac{1}{2}})\|\nabla_h w_m\|_2^{\frac{1}{2}}\Big]\Big(\int_0^T  \|\Delta_h w_m\|_2^2  dt\Big)^{\frac{1}{4}} \Big(\int_0^T \|\eta\|_2^{\frac{4}{3}}dt\Big)^{\frac{3}{4}},
\end{align}
which is uniformly bounded due to (\ref{Gal-2}) and (\ref{Gal-3}).
Consequently, the sequence $\mathbf u_m \cdot \nabla_h w_m$ is bounded in $L^4(0,T;L^2(\Omega))$.
As a result, we obtain from equation (\ref{Gl-1}) that
\begin{align}    \label{Gal-6}
\partial_t w_m \;\;\text{is uniformly bounded in}\;\; L^2(\Omega \times (0,T)).
\end{align}

Next we show that $\partial_t \mathbf u_m $ is bounded in $L^2(\Omega \times (0,T))$. For any function $\tilde{\eta} \in L^2(0,T;H^1_h(\Omega))$, we apply Lemma \ref{lemma1} and Poincar\'e inequality (\ref{poin}) to estimate
\begin{align}     \label{Gal-7}
&\int_0^T \int_{\Omega}  (\mathbf u_m \cdot \nabla_h \omega_m) \tilde{\eta} dx dy dz dt \notag\\
&\leq  C\int_0^T \|\nabla_h \mathbf u_m\|_2^{1/2} \left(\|\mathbf u_m\|_2 + \|\partial_z \mathbf u_m\|_2\right)^{1/2} \|\nabla_h \omega_m\|_2
\|  \tilde{\eta}   \|_2^{1/2}    \left(\|   \tilde{\eta}    \|_2 + \|\nabla_h    \tilde{\eta}     \|_2\right)^{1/2}   dt \notag\\
&\leq C\sup_{t\in[0,T]} \Big[ \|\nabla_h \mathbf u_m\|_2^{\frac{1}{2}} (\|\mathbf u_m\|_2^{\frac{1}{2}} + \|\partial_z\mathbf u_m\|_2^{\frac{1}{2}})\Big]
\Big(\int_0^T \|\nabla_h \omega_m\|_2^2 dt \Big)^{\frac{1}{2}} \Big(\int_0^T \|\tilde{\eta}    \|_2^2 + \|\nabla_h   \tilde{\eta}     \|_2^2  dt\Big)^{\frac{1}{2}}.
\end{align}
Note that (\ref{Gal-2})-(\ref{Gal-3}) provide the uniform bound for the right-hand side of (\ref{Gal-7}). Therefore, the sequence
$\mathbf u_m \cdot \nabla_h \omega_m$ is bounded uniformly in $m$ in $L^2(0,T;(H^1_h(\Omega))')$, where $(H^1_h(\Omega))'$ is the dual space of $H^1_h(\Omega)$.
Consequently, we obtain from the vorticity equation (\ref{Gl-2}) that
\begin{align}
&\partial_t \omega_m  \;\;\text{is uniformly bounded in}\;\;  L^2(0,T;(H^1_h(\Omega))');  \label{Gal-88}\\
&\partial_t \mathbf u_m  \;\;\text{is uniformly bounded in}\;\;  L^2(\Omega \times (0,T))  \label{Gal-8} .
\end{align}

Moreover, $\partial_t {\theta^{(m)}}$ is bounded in $L^2(\Omega \times (0,T))$. Indeed, applying H\"older's inequality, we deduce
\begin{align}  \label{Gal-10}
\int_{\Omega} |w_m\bar{w_m{\theta^{(m)}}}|^2 dx dy dz  &\leq  C\int_0^1 \left(\int_{[0,L]^2} |w_m|^2 dx dy \right)^2  \left(\int_{[0,L]^2} |{\theta^{(m)}}|^2 dx dy \right) dz \notag\\
&\leq C   \int_0^1 \left(\int_{[0,L]^2} |w_m|^6 dx dy \right)^{2/3}  \left(\int_{[0,L]^2} |{\theta^{(m)}}|^6 dx dy \right)^{1/3} dz   \notag\\
&\leq C \|w_m\|_6^4  \|{\theta^{(m)}}\|_6^2    \leq C \|w_m\|_{H^1}^4  \|{\theta^{(m)}}\|_{H^1}^2,
\end{align}
where the last inequality is due to the imbedding $H^1(\Omega) \hookrightarrow L^6(\Omega)$ in three dimensions. Since the $H^1$ norms of $w_m$ and ${\theta^{(m)}}$ are uniformly bounded on $[0,T]$, (\ref{Gal-10}) implies
the sequence $w_m\bar{w_m{\theta^{(m)}}}$ is bounded in $L^{\infty}(0,T;L^2(\Omega))$. Also, using an estimate similar to (\ref{Gal-5}), we can show the sequence
$\mathbf u_m \cdot \nabla_h {\theta^{(m)}}$ is bounded in $L^4(0,T;L^2(\Omega))$. Therefore, we obtain from the temperature equation (\ref{model-3}) that
\begin{align}  \label{Gal-11}
\partial_t {\theta^{(m)}} \;\;\text{is uniformly bounded in}\;\; L^2(\Omega \times (0,T)).
\end{align}

Owing to (\ref{Gal-6}), (\ref{Gal-88})-(\ref{Gal-8}) and (\ref{Gal-11}), on a subsequence,
\begin{align}  \label{Gal-111}
&\partial_t w_m \rightarrow \partial_t w, \;\;  \partial_t \mathbf u_m \rightarrow \partial_t \mathbf u, \;\; \partial_t {\theta^{(m)}} \rightarrow \partial_t \theta
\;\;\; \text{weakly in} \;\; L^2(\Omega \times (0,T)); \\
& \partial_t \omega_m \rightarrow \partial_t \omega   \;\;\;      \text{weakly$^*$ in} \;\; L^2(0,T;(H^1_h(\Omega))').    \label{Gal-112}
\end{align}

By (\ref{Gal-2}), (\ref{Gal-6}), (\ref{Gal-8}), (\ref{Gal-11}), and thanks to the Aubin's compactness theorem, the following strong convergence holds for a subsequence
of $(\mathbf u_m, w_m, {\theta^{(m)}})^{tr}$:
\begin{align}     \label{Gal-12}
\mathbf u_m \rightarrow \mathbf u, \;\; w_m\rightarrow w, \;\; {\theta^{(m)}} \rightarrow \theta   \;\; \text{in} \;\; L^2(\Omega \times (0,T)).
\end{align}
Thus, $\omega_m \rightarrow \omega$ in $L^2(0,T;(H^1_h(\Omega))')$ for this subsequence.

Now we can pass to the limit as $m\rightarrow \infty$ for the nonlinear terms in the Galerkin-like system (\ref{Gl-1})-(\ref{Gl-5}). Let $\psi$ be a trigonometric polynomial with continuous coefficients.
For $m$ larger than the degree of $\psi$, we have
\begin{align}    \label{Gal-13}
&\int_0^T \int_{\Omega} P_m(\mathbf u_m \cdot \nabla_h \omega_m) \psi dx dy dz dt  \notag\\
&= \int_0^T \int_{\Omega} (\mathbf u \cdot \nabla_h \omega_m) \psi dx dy dz dt + \int_0^T \int_{\Omega} [(\mathbf u_m-\mathbf u) \cdot \nabla_h \omega_m] \psi dx dy dz dt.
\end{align}
Since $\nabla_h \omega_m\rightarrow \nabla_h \omega$ weakly in $L^2(\Omega \times (0,T))$, $\mathbf u_m\rightarrow \mathbf u$ in $L^2(\Omega \times (0,T))$,
and $\nabla_h \omega_m$ is bounded in $L^2(\Omega \times (0,T))$, we can pass to the limit in (\ref{Gal-13}) to get
\begin{align}     \label{Gal-14}
\lim_{m\rightarrow \infty}  \int_0^T \int_{\Omega} P_m(\mathbf u_m \cdot \nabla_h \omega_m) \psi dx dy dz dt
=  \int_0^T \int_{\Omega} (\mathbf u \cdot \nabla_h \omega) \psi dx dy dz dt.
\end{align}
Similarly, we can deduce
\begin{align}
&\lim_{m\rightarrow \infty}  \int_0^T \int_{\Omega} P_m(\mathbf u_m \cdot \nabla_h w_m) \psi dx dy dz dt
=  \int_0^T \int_{\Omega} (\mathbf u \cdot \nabla_h w) \psi dx dy dz dt.     \label{Gal-15}\\
& \lim_{m\rightarrow \infty}  \int_0^T \int_{\Omega} (\mathbf u_m \cdot \nabla_h {\theta^{(m)}})  \psi dx dy dz dt
=  \int_0^T \int_{\Omega} (\mathbf u \cdot \nabla_h \theta) \psi dx dy dz dt.    \label{Gal-16}
\end{align}

Furthermore, we consider
\begin{align}        \label{Gal-17}
&\int_0^T \int_{\Omega} \left( w_m \bar{w_m {\theta^{(m)}}} -w \bar{w\theta} \right)\psi dx dy dz dt \notag\\
&= \int_0^T \int_{\Omega} (w_m - w)  \bar{w_m {\theta^{(m)}}}  \psi dx dy dz dt
+ \int_0^T \int_{\Omega} w \bar{(w_m-w) {\theta^{(m)}}} \psi dx dy dz dt    \notag\\
&  \hspace{0.5 in} + \int_0^T \int_{\Omega} w \bar{w ({\theta^{(m)}}-\theta)} \psi dx dy dz dt.
\end{align}
We shall show that each integral on the right-hand side of (\ref{Gal-17}) converges to zero. The convergence to zero for the first integral on the right-hand side of (\ref{Gal-17}) is due to the fact that $w_m\rightarrow w$ in $L^2(\Omega \times (0,T))$ and the uniform boundedness of the sequence $\bar{w_m{\theta^{(m)}}}$ in $L^2(\Omega \times (0,T))$. For the second integral on the right-hand side of (\ref{Gal-17}), we apply Cauchy-Schwarz inequality and Lemma \ref{lemma2} to get
\begin{align*}
&\left|\int_0^T \int_{\Omega} w \bar{(w_m-w) {\theta^{(m)}}} \psi dx dy dz dt \right|\notag\\
&\leq C  \|\psi\|_{L^{\infty}(\Omega \times (0,T))}  \int_0^T \int_0^1  \Big(\int_{[0,L]^2} {  |\theta^{(m)}}|^2 dx dy\Big)^{\frac{1}{2}}  \Big(\int_{[0,L]^2} |w_m - w|^2 dx dy \Big)^{\frac{1}{2}}
\Big(\int_{[0,L]^2} |w| dxdy  \Big) dz dt   \notag\\
&\leq C  \|\psi\|_{L^{\infty}(\Omega\times (0,T))} \sup_{t\in [0,T]}\left(\|{\theta^{(m)}}\|_2+ \|\partial_z {\theta^{(m)}}\|_2\right)
 \|w_m-w\|_{L^2(\Omega \times (0,T))} \|w\|_{L^2(\Omega \times (0,T))} \longrightarrow 0,
\end{align*}
where the convergence to zero is due to the fact that $w_m\rightarrow w$ in $L^2(\Omega \times (0,T))$ and the sequence ${\theta^{(m)}}$ is uniformly bounded in $L^{\infty}(0,T;H^1(\Omega))$.
Next, we look at the last term on the right-hand side of (\ref{Gal-17}):
\begin{align*}
&\left|\int_0^T \int_{\Omega} w \bar{w ({\theta^{(m)}}-\theta)} \psi dx dy dz dt  \right| \notag\\
&\leq   C  \|\psi\|_{L^{\infty}(\Omega \times (0,T))} \int_0^T \int_0^1 \left(\int_{[0,L]^2} w^4 dx dy\right)^{1/2} \left(\int_{[0,L]}|{\theta^{(m)}} - \theta|^2 dx dy\right)^{1/2} dz dt\notag\\
&\leq   C  \|\psi\|_{L^{\infty}(\Omega \times (0,T))}   \sup_{t\in [0,T]}\|w\|_4^2   \|{\theta^{(m)}} - \theta\|_{L^2(\Omega \times (0,T))}\longrightarrow 0,
\end{align*}
where the convergence to zero is due to the fact that ${\theta^{(m)}}\rightarrow \theta$ in $L^2(\Omega \times (0,T))$
and that $w\in L^{\infty}([0,T];H^1(\Omega))$. In sum, all integrals on the right-hand side of (\ref{Gal-17}) converge to zero, and thus
\begin{align}   \label{Gal-21}
\lim_{m\rightarrow \infty}\int_0^T \int_{\Omega} (w_m \bar{w_m {\theta^{(m)}}} )\psi dx dy dz dt =  \int_0^T \int_{\Omega} w \bar{w \theta} \psi dx dy dz dt.
\end{align}

Owing to (\ref{Gal-5'})-(\ref{Gal-5'''}), (\ref{Gal-111})-(\ref{Gal-112}), (\ref{Gal-14})-(\ref{Gal-16}), (\ref{Gal-21}), we can pass to the limit as $m\rightarrow \infty$ for the Galerkin-like system (\ref{Gl-1})-(\ref{Gl-5}) to get
\begin{align}    \label{Gal-22}
\begin{cases}
\int_0^T \int_{\Omega} \left(\partial_t w + \mathbf u \cdot \nabla_h w - \partial_z \phi -\Gamma \theta - \frac{1}{Re} \Delta_h w \right) \psi dx dy dz dt=0,  \\
\int_0^T \int_{\Omega} \left(\partial_t \omega + \mathbf u \cdot \nabla_h \omega - \partial_z w -\frac{1}{Re} \Delta_h \omega  - \epsilon^2 \partial_{zz} \phi \right)\psi dx dy dz dt=0, \\
\int_0^T \int_{\Omega} \left(\partial_t \theta  + \mathbf u \cdot \nabla_h \theta +  w \bar{w\theta} - \frac{1}{Pe} \Delta_h \theta\right) \psi dx dy dz dt=0,\\
\end{cases}
\end{align}
such that $\nabla_h \cdot \mathbf u=0$, for any trigonometric polynomial $\psi$ with continuous coefficients.

By using estimates similar to (\ref{Gal-5}) and (\ref{Gal-7}), one has $\mathbf u \cdot \nabla_h w$, $\mathbf u \cdot \nabla_h \theta \in L^4(0,T;L^2(\Omega))$ and $\mathbf u \cdot \nabla_h \omega \in L^2(0,T;(H^1_h(\Omega))')$ . Also, $w\bar{w\theta}\in L^{\infty}(0,T;L^2(\Omega))$ due to an estimate like (\ref{Gal-10}).
Hence, we obtain from (\ref{Gal-22}) that
equations (\ref{model-1})-(\ref{model-3}) hold in the sense of (\ref{main-1}). By simply integrating (\ref{main-1}) in time, we see that $w, \mathbf u, \theta \in C([0,T];L^2(\Omega))$ and $\omega \in C([0,T];(H^1_h(\Omega))')$.
Then, it is easy to verify the initial condition.
Also, by (\ref{mean0}) and (\ref{Gal-12}), we find that $\bar{\mathbf u}=0$, $\bar w =0$ and $\bar \theta=0$ for all $t\in [0,T]$.
Finally, due to the regularity of solutions, we can multiply (\ref{main-1}) by $(w,\phi,\theta)^{tr}$ and integrate the result over $\Omega \times [0,t]$ for $t\in [0,T]$ to obtain the energy identity (\ref{energy}). This completes the proof for the global existence of strong solutions for system (\ref{model-1})-(\ref{model-4}).

\vspace{0.1 in}

\section{Uniqueness of strong solutions and continuous dependence on initial data}  \label{sec-unique}
This section is devoted to proving that the strong solutions for system (\ref{model-1})-(\ref{model-4}) are unique and depend continuously on initial data.
Since a strong solution has the regularity specified in (\ref{main-2}), all calculations in this section are legitimate.
Assume there are two strong solutions $(\mathbf u_1,w_1,\theta_1)^{tr}$ and $(\mathbf u_2,w_2,\theta_2)^{tr}$ for system (\ref{model-1})-(\ref{model-4}).
Let $\omega_1 = \nabla_h \times \mathbf u_1$ and $\omega_2 = \nabla_h \times \mathbf u_2$.
Set $\mathbf u=\mathbf u_1-\mathbf u_2$, $\omega=\omega_1-\omega_2$, $w=w_1-w_2$, $\theta=\theta_1-\theta_2$. Therefore, for a.e. $t\in [0,T]$,
\begin{align}     \label{unique}
\begin{cases}
w_t + \mathbf u \cdot \nabla_h w_1+ \mathbf u_2 \cdot \nabla_h w - \phi_z= \Gamma \theta + \frac{1}{Re} \Delta_h w,  \;\;  \text{in}\; L^2(\Omega),\\
\omega_t + \mathbf u \cdot \nabla_h \omega_1 + \mathbf u_2 \cdot \nabla_h \omega - w_z = \frac{1}{Re} \Delta_h \omega  + \epsilon^2 \phi_{zz},
\;\;\text{in}\; (H^1_h(\Omega))',\\
\theta_t  + \mathbf u \cdot \nabla_h \theta_1 +  \mathbf u_2 \cdot \nabla_h \theta +  w \bar{w_1\theta_1} + w_2 \bar{w\theta_1} +   w_2 \bar{w_2\theta}
= \frac{1}{Pe} \Delta_h \theta,    \;\;\text{in} \; L^2(\Omega),
\end{cases}
\end{align}
and $\nabla_h \cdot \mathbf u =0$.

We multiply (\ref{unique}) by $(w,\phi,\theta)^{tr}$ and integrate over $\Omega$. By using (\ref{iden-0}), (\ref{iden-1}), (\ref{iden-3}) and (\ref{iden-2}), we obtain, for a.e. $t\in [0,T]$,
\begin{align}   \label{unique-1}
&\frac{1}{2}\frac{d}{dt}\left(\|w\|_2^2 + \|\mathbf u\|_2^2   +  \|\theta\|_2^2 \right)
+ \frac{1}{Re}\left(\|\nabla_h w\|_2^2 + \|\nabla_h \mathbf u\|_2^2\right)
+\frac{1}{Pe} \|\nabla_h \theta\|_2^2 + \epsilon^2\|\phi_z\|_2^2
+ \|\bar{w_2 \theta}\|_2^2   \notag\\
&\leq \int_{\Omega} |(\mathbf u \cdot \nabla_h w) w_1| dx dy dz
+\int_{\Omega} |(\mathbf u_2 \cdot \nabla_h \phi) \omega| dx dy dz +
\int_{\Omega} |(\mathbf u \cdot \nabla_h \theta) \theta_1| dx dy dz  \notag\\
&\hspace{0.3 in} + \frac{\Gamma}{2}\left(\|\theta\|_2^2 + \|w\|_2^2\right) + \int_{\Omega} |w \bar{w_1\theta_1} \theta| dx dy dz
-\int_{\Omega} w_2 \bar{w\theta_1} \theta dx dy dz.
\end{align}
Now we estimate each term on the right-hand side of (\ref{unique-1}).

Using Lemma \ref{lemma1} with $f=w_1$, $g=\mathbf u$ and $h=\nabla_h w$, and by Poincar\'e inequality (\ref{poin}), we obtain
\begin{align}    \label{unique-2}
&\int_{\Omega} |(\mathbf u \cdot \nabla_h w) w_1| dx dy dz   \notag\\
&\leq C \|\nabla_h w_1\|_2^{1/2}  \left(\|w_1\|_2 + \|\partial_z w_1\|_2\right)^{1/2}
\|\mathbf u\|_2^{1/2} \|\nabla_h \mathbf u\|_2^{1/2} \|\nabla_h w\|_2 \notag\\
&\leq \frac{1}{6Re} \left(\|\nabla_h w\|_2^2 + \|\nabla_h \mathbf u\|_2^2\right)+
C\|\nabla_h w_1\|_2^2   \left(\|w_1\|_2^2 + \|\partial_z w_1\|_2^2\right) \|\mathbf u\|_2^2.
\end{align}
Also, using Lemma \ref{lemma1} with $f=\mathbf u_2$, $g=\nabla_h \phi$, $h=\omega$, and by Poincar\'e inequality (\ref{poin}), we have
\begin{align}  \label{unique-3}
\int_{\Omega} |(\mathbf u_2 \cdot \nabla_h \phi) \omega| dx dy dz
&\leq C    \|\omega_2\|_2^{1/2}        (\|\mathbf u_2\|_2+ \|\partial_z \mathbf u_2\|_2)^{1/2} \|\mathbf u\|_2^{1/2} \|\nabla_h \mathbf u\|_2^{3/2}  \notag\\
&\leq  \frac{1}{6Re} \|\nabla_h \mathbf u\|_2^2 +   C    \|\omega_2\|_2^2      (\|\mathbf u_2\|_2^2+ \|\partial_z \mathbf u_2\|_2^2)\|\mathbf u\|_2^2.
\end{align}
Moreover, by Lemma \ref{lemma1} with $f=\theta_1$, $g=\mathbf u$, $h=\nabla_h \theta$, and Poincar\'e inequality (\ref{poin}), one has
\begin{align}    \label{unique-4}
&\int_{\Omega} |(\mathbf u \cdot \nabla_h \theta) \theta_1| dx dy dz  \notag\\
&\leq  \|\nabla_h \theta_1\|_2^{1/2}     \left(\|\theta_1\|_2  + \|\partial_z \theta_1\|_2 \right)^{1/2}
\|\mathbf u\|_2^{1/2}    \|\nabla_h \mathbf u\|_2^{1/2}   \|\nabla_h \theta\|_2  \notag\\
&\leq \frac{1}{6Re} \|\nabla_h \mathbf u\|_2^2 + \frac{1}{2Pe} \|\nabla_h \theta\|_2^2 +
C  \|\nabla_h \theta_1\|_2^2    \left(\|\theta_1\|_2^2  + \|\partial_z \theta_1\|_2^2 \right) \|\mathbf u\|_2^2.
\end{align}

Using Cauchy-Schwarz inequality and Lemma \ref{lemma2}, we get
\begin{align}  \label{unique-5}
\int_{\Omega} |w \bar{w_1 \theta_1} \theta| dx dy dz
&\leq  C\|w\|_2 \|\theta\|_2   \sup_{z\in [0,1]}   \Big[\Big(\int_{[0,L]^2} w_1^2 dx dy\Big)^{1/2} \Big(\int_{[0,L]^2} \theta_1^2 dx dy\Big)^{1/2} \Big] \notag\\
&\leq    C\|w\|_2 \|\theta\|_2     (\|w_1\|_2+\|\partial_z w_1\|_2)   (\|\theta_1\|_2+\|\partial_z \theta_1\|_2)  \notag\\
&\leq C (\|w_1\|_2^2+\|\partial_z w_1\|_2^2)  \|w\|_2^2 +   (\|\theta_1\|_2^2+\|\partial_z \theta_1\|_2^2)    \|\theta\|_2^2.
\end{align}
Again, applying Cauchy-Schwarz inequality and Lemma \ref{lemma2}, we obtain
\begin{align}    \label{unique-6}
-\int_{\Omega} w_2 \bar{w\theta_1} \theta dx dy dz
&\leq \int_0^1 \Big(\int_{[0,L]^2} w^2 dx dy\Big)^{1/2}   \Big(\int_{[0,L]^2} \theta_1^2 dx dy\Big)^{1/2}   \left|\bar{w_2 \theta}\right| dz  \notag\\
&\leq  C   (\|\theta_1\|_2+\|\partial_z \theta_1\|_2)   \|w\|_2 \|\bar{w_2 \theta}\|_2    \notag\\
&\leq  \|\bar{w_2 \theta}\|_2^2 + C(\|\theta_1\|_2^2+\|\partial_z \theta_1\|_2^2) \|w\|_2^2.
\end{align}

Now, we combine estimates (\ref{unique-1})-(\ref{unique-6}) to deduce, for a.e. $t\in [0,T]$,
\begin{align*}
&\frac{d}{dt}\left(\|w\|_2^2 + \|\mathbf u\|_2^2   +  \|\theta\|_2^2 \right)
+ \frac{1}{Re}\left(\|\nabla_h w\|_2^2 + \|\nabla_h \mathbf u\|_2^2\right)
+\frac{1}{Pe} \|\nabla_h \theta\|_2^2 + \epsilon^2\|\phi_z\|_2^2  \notag\\
&\leq  C  \Big[\|\nabla_h w_1\|_2^2  \left(\|w_1\|_2^2 + \|\partial_z w_1\|_2^2\right)+\|\omega_2\|_2^2   (\|\mathbf u_2\|_2^2+ \|\partial_z \mathbf u_2\|_2^2)
+\|\nabla_h \theta_1\|_2^2    \left(\|\theta_1\|_2^2  + \|\partial_z \theta_1\|_2^2 \right) \Big]\|\mathbf u\|_2^2   \notag\\
 &\hspace{0.2 in}+ C\left( \|w_1\|_2^2+\|\partial_z w_1\|_2^2
+\|\theta_1\|_2^2+\|\partial_z \theta_1\|_2^2 +1\right) \|w\|_2^2
+ (\|\theta_1\|_2^2+\|\partial_z \theta_1\|_2^2 +1 )   \|\theta\|_2^2.
\end{align*}
Thanks to the Gronwall's inequality, we have, for all $t\in [0,T]$,
\begin{align}   \label{unique-7}
\|w(t)\|_2^2 + \|\mathbf u(t)\|_2^2   +  \|\theta(t)\|_2^2
\leq  \left(\|w(0)\|_2^2 + \|\mathbf u(0)\|_2^2   +  \|\theta(0)\|_2^2\right)  e^{K(t)}
\end{align}
where
\begin{align*}
&K(t) = C \int_0^t   \Big[\left(\|\nabla_h w_1\|_2^2 +1 \right)  \left(\|w_1\|_2^2 + \|\partial_z w_1\|_2^2  \right)+\|\omega_2\|_2^2   (\|\mathbf u_2\|_2^2+ \|\partial_z \mathbf u_2\|_2^2)  \\
&   \hspace{0.5 in}+(\|\nabla_h \theta_1\|_2^2+1)    \left(\|\theta_1\|_2^2  + \|\partial_z \theta_1\|_2^2 \right) +1  \Big] ds.
\end{align*}
Since strong solutions $(\mathbf u_1,w_1,\theta_1)^{tr}$ and $(\mathbf u_2,w_2,\theta_2)^{tr}$ are in the space $L^{\infty}(0,T;(H^1(\Omega))^4)$, $K(t)$ is bounded on $[0,T]$. Therefore, (\ref{unique-7}) implies the uniqueness of strong solutions. Furthermore, (\ref{unique-7}) also implies the continuous dependence on initial data for strong solutions, namely,
if $\{(\mathbf u_0^n,w_0^n, \theta_0^n)^{tr}\}$ is a bounded sequence of initial data in $H^1(\Omega)$
such that $(\mathbf u_0^n,w_0^n,\theta_0^n)^{tr} \rightarrow (\mathbf u_0,w_0,\theta_0)^{tr}$ with respect to the $L^2(\Omega)$ norm, then the corresponding strong solutions $(\mathbf u^n,w^n,\theta^n)^{tr} \rightarrow (\mathbf u,w,\theta)^{tr}$ in $C([0,T];(L^2(\Omega))^4)$.

\vspace{0.1 in}

\section{Large-time behavior}  \label{sec-large}
In this section, we prove Theorem \ref{thm-decay}: the asymptotic behavior of strong solutions as $t\rightarrow \infty$.
Since a strong solution has the regularity specified in (\ref{main-2}), all calculations in this section are legitimate.

First we show the exponential decay estimates (\ref{thm2-00})-(\ref{thm2-1}).
Taking the inner product of (\ref{model-3}) with $\theta$ gives
\begin{align}  \label{theta-la}
\frac{1}{2}\frac{d}{dt}\|\theta\|_2^2   +  \frac{1}{Pe} \|\nabla_h \theta\|_2^2   + \|\bar{w\theta}\|_2^2  =0.
\end{align}

According to the Poincar\'e inequality $\|\theta\|_2^2  \leq \gamma \|\nabla_h \theta\|_2^2 $ and estimate (\ref{theta-la}), we obtain
\begin{align}    \label{large-1}
\frac{1}{2}\frac{d}{dt} \|\theta\|_2^2   +  \frac{1}{\gamma Pe} \|\theta\|_2^2  \leq 0.
\end{align}
It follows that
\begin{align}    \label{large-2}
\|\theta(t)\|_2^2  \leq e^{-\frac{2}{\gamma Pe}t} \|\theta_0\|_2^2,   \;\;\;  \text{for all}\;\;   t\geq 0.
\end{align}
Then, integrating (\ref{theta-la}) over $[t,t+1]$ gives
\begin{align}   \label{large-2'}
\int_t^{t+1} \left(  \frac{1}{Pe}    \|\nabla_h \theta\|_2^2  + \|\bar{w\theta}\|_2^2\right)  ds
\leq \frac{1}{2}\|\theta(t)\|_2^2 \leq  \frac{1}{2} e^{-\frac{2}{\gamma Pe}t} \|\theta_0\|_2^2,  \;\;\;  \text{for all}\;\;   t\geq 0.\end{align}

Taking the $L^2(\Omega)$ inner product of equations (\ref{model-1})-(\ref{model-2}) with $(w,\phi)^{tr}$, we deduce
\begin{align}    \label{wu-la}
\frac{d}{dt} \left(\|w\|_2^2 + \|\mathbf u\|_2^2 \right) +\frac{1}{Re} \left(\|\nabla_h w\|_2^2
+ \|\nabla_h \mathbf u\|_2^2 \right)  + \epsilon^2 \|\partial_z \phi\|_2^2  \leq C \|\theta\|_2^2.
\end{align}
Integrating (\ref{wu-la}) over $[0,t]$ gives
\begin{align}   \label{lar-1}
&\|w(t)\|_2^2 + \|\mathbf u(t)\|_2^2 +
\int_0^t  \left(\frac{1}{Re} \left(\|\nabla_h w\|_2^2
+ \|\nabla_h \mathbf u\|_2^2 \right)  + \epsilon^2 \|\partial_z \phi\|_2^2 \right) ds  \notag\\
&\leq      \|w_0\|_2^2 + \|\mathbf u_0\|_2^2  + C \|\theta_0\|_2^2,  \;\;\; \text{for all} \;\;t\geq 0,
 \end{align}
by virtue of (\ref{large-2}).

Applying the Poincar\'e inequality to (\ref{wu-la}), we have
\begin{align}  \label{large-3}
\frac{d}{dt} \left( \|w\|_2^2  + \|\mathbf u\|_2^2 \right) + \frac{1}{\kappa \gamma  Re} \left( \|w\|_2^2  + \|\mathbf u\|_2^2 \right) \leq C \|\theta\|_2^2,
\end{align}
for any $\kappa \geq 1$.
From (\ref{large-3}) and (\ref{large-2}), we obtain
\begin{align}   \label{large-4}
\frac{d}{dt}\left(e^{\frac{1}{\kappa \gamma Re}t}\left(\|w\|^2_2 +  \|\mathbf u\|_2^2 \right)    \right)
\leq C e^{\frac{1}{\kappa \gamma Re} t}  \|\theta\|_2^2   \leq Ce^{ \left(\frac{1}{\kappa \gamma Re} -\frac{2}{\gamma Pe} \right)  t}  \|\theta_0\|_2^2.
\end{align}
We can choose $\kappa \geq 1$ such that $Pe \not = 2\kappa Re$. Then integrating (\ref{large-4}) over $[0,t]$ implies
\begin{align}   \label{large-5}
\|w(t)\|_2^2 +  \|\mathbf u(t)\|_2^2  \leq    e^{-\frac{1}{\kappa \gamma Re}t}  \left( \|w_0\|_2^2  +  \|\mathbf u_0\|_2^2 \right)
+C   \left( e^{-\frac{2}{\gamma Pe}t}   +  e^{-\frac{1}{\kappa \gamma Re}t} \right) \|\theta_0\|_2^2,
\end{align}
for all $t\geq 0$.
Integrating (\ref{wu-la}) from $t$ to $t+1$ shows
\begin{align}   \label{large-6}
&\int_t^{t+1} \left(\frac{1}{Re}\left(\|\nabla_h w\|_2^2 + \|\nabla_h \mathbf u\|_2^2  \right) + \epsilon^2 \|\phi_z\|^2_2  \right)
ds    \leq  \|w(t)\|_2^2 +  \|\mathbf u(t)\|_2^2   +  C\int_t^{t+1} \|\theta\|_2^2 ds   \notag\\
&\leq  e^{-\frac{1}{\kappa \gamma Re}t}     \left(\|w_0\|_2^2  + \|\mathbf u_0\|_2^2 \right) +  C\left(e^{-\frac{2}{\gamma Pe}t} + e^{-\frac{1}{\kappa \gamma Re}t} \right) \|\theta_0\|_2^2, \;\;\;  \text{for all}\,\, t\geq 0,
\end{align}
where the last inequality is due to (\ref{large-2}) and (\ref{large-5}).

Next, we take the inner product of (\ref{model-2}) with $\omega$, and adopt the same calculation as in (\ref{omega-1})-(\ref{omega-3}) to derive
\begin{align}   \label{omega-la}
\frac{d}{dt} \|\omega\|_2^2
+ \frac{2}{Re}\|\nabla_h \omega\|_2^2 + \epsilon^2 \|\mathbf u_z\|_2^2
\leq \frac{1}{\epsilon^2} \|\nabla_h w\|_2^2.
\end{align}
Integrating (\ref{omega-la}) over $[0,t]$ and using (\ref{lar-1}), we obtain
\begin{align}  \label{lar-2}
\|\omega(t)\|_2^2 +  \int_0^t  \left(\frac{2}{Re}\|\nabla_h \omega\|_2^2 + \epsilon^2 \|\mathbf u_z\|_2^2 \right) ds  \leq  \|\omega_0\|_2^2
+  C (\|w_0\|_2^2 + \|\mathbf u_0\|_2^2  + \|\theta_0\|_2^2),
\end{align}
for all $t\geq 0$.

Now, we integrate (\ref{omega-la}) over $[s,t+1]$ for $t\leq s\leq t+1$, and then integrate the result from $t$ to $t+1$ to get
\begin{align}   \label{large-7}
\|\omega(t+1)\|_2^2 & \leq \int_t^{t+1} \|\omega\|_2^2 ds  +  \frac{1}{\epsilon^2} \int_t^{t+1} \|\nabla_h w\|_2^2 ds  \notag\\
&\leq  C e^{-\frac{1}{\kappa \gamma Re}t}     \left(\|w_0\|_2^2  + \|\mathbf u_0\|_2^2 \right) +  C\left(e^{-\frac{2}{\gamma Pe}t} + e^{-\frac{1}{\kappa \gamma Re}t} \right) \|\theta_0\|_2^2,
\end{align}
for all $t\geq 0$, by using (\ref{large-6}).

Then, integrating (\ref{omega-la}) over $[t,t+1]$ shows
\begin{align}   \label{large-8}
&\int_t^{t+1}  \left(\frac{2}{Re} \|\nabla_h \omega\|_2^2 +  \epsilon^2 \|\mathbf u_z\|_2^2\right) ds
\leq \|\omega(t)\|_2^2  +  \frac{1}{\epsilon^2}\int_t^{t+1} \|\nabla_h w\|_2^2 ds   \notag\\
&\leq C e^{-\frac{1}{\kappa \gamma Re}t}     \left(\|w_0\|_2^2  + \|\mathbf u_0\|_2^2 \right) +  C\left(e^{-\frac{2}{\gamma Pe}t} + e^{-\frac{1}{\kappa \gamma Re}t} \right) \|\theta_0\|_2^2,      \;\;\;  \text{for all}\,\, t\geq 1,
\end{align}
where the last inequality is due to (\ref{large-6}) and (\ref{large-7}).

Furthermore, taking the inner product of (\ref{model-1}) with $-\Delta_h w$ and using the calculations in (\ref{H1w-00})-(\ref{H1w}) yield
\begin{align}     \label{nablaw-la}
\frac{d}{dt}\|\nabla_h w\|_2^2 +\frac{1}{Re} \|\Delta_h w\|_2^2   \leq C\|\omega\|_2^2
(\|\mathbf u\|_2^2 + \|\mathbf u_z\|_2^2) \|\nabla_h w\|_2^2 + C\left(\|\phi_z\|_2^2
+ \|\theta\|_2^2\right),
\end{align}
and
\begin{align}   \label{lar-3}
\|\nabla_h w(t)\|_2^2 +  \frac{1}{Re} \int_0^t  \|\Delta_h w\|_2^2 ds
\leq C (\|\nabla_h w_0\|_2, \|\omega_0\|_2, \|\theta_0\|_2), \;\;  \text{for all}\,\, t\geq 0.
\end{align}

Applying the uniform Gronwall Lemma (see Lemma \ref{gronwall}) to (\ref{nablaw-la}), we obtain
\begin{align}      \label{large-9}
\|\nabla_h w(t+1)\|_2^2
&\leq e^{\int_t^{t+1} C \|\omega\|_2^2 (\|\mathbf u\|_2^2 + \|\mathbf u_z\|_2^2)  ds  }
\left( \int_t^{t+1}  \|\nabla_h w\|_2^2 ds +   C\int_t^{t+1}  \left( \|\phi_z\|_2^2  +  \|\theta\|_2^2   \right)ds      \right)   \notag\\
&\leq   \left(e^{-\frac{2}{\gamma Pe}t} + e^{-\frac{1}{\kappa \gamma Re}t} \right)  C(\|\mathbf u_0\|_2,  \|w_0\|_2,  \|\theta_0\|_2),
\;\;\;  \text{for all}\,\, t\geq 1,
\end{align}
where we use (\ref{large-2}), (\ref{large-5})-(\ref{large-6}) and (\ref{large-7})-(\ref{large-8}). Then, we integrate (\ref{nablaw-la}) over $[t,t+1]$ to get
\begin{align}    \label{large-10}
\frac{1}{Re}\int_t^{t+1} \|\Delta_h w\|_2^2 ds
&\leq   \|\nabla_h w(t)\|_2^2
+ C\int_t^{t+1} \left[\|\omega\|_2^2(\|\mathbf u\|_2^2 + \|\mathbf u_z\|_2^2) \|\nabla_h w\|_2^2 + \|\phi_z\|_2^2 + \|\theta\|_2^2 \right] ds  \notag\\
&\leq    \left(e^{-\frac{2}{\gamma Pe}t} + e^{-\frac{1}{\kappa \gamma Re}t} \right)  C(\|\mathbf u_0\|_2,  \|w_0\|_2,  \|\theta_0\|_2),
\;\;\;  \text{for all}\,\, t\geq 2,
\end{align}
due to (\ref{large-2}) and (\ref{large-5})-(\ref{large-6}), (\ref{large-7})-(\ref{large-8}) and (\ref{large-9}).

Next, we multiply (\ref{model-3}) by $2\theta$ and integrate the result over $[0,L]^2 \times [0,t]$, then we obtain
\begin{align}   \label{theta-inf1}
\frac{1}{2}\bar{\theta^2}(z,t)+\int_0^t  \left[\frac{1}{Pe} \bar{|\nabla_h \theta|^2}(z) +   \left|\bar{w\theta}\right|^2 (z) \right] ds = \frac{1}{2}\bar{\theta_0^2}(z)
\leq C\left( \|\theta_0\|_2^2 +  \|\partial_z \theta_0\|_2^2 \right),
\end{align}
for all $t\geq 0$, and for a.e. $z\in [0,1]$, where we use Lemma \ref{lemma2} to obtain the last inequality.

Also, taking the inner product of (\ref{model-3}) with $-\Delta_h \theta$ and using the same calculation as in (\ref{H1ta-f1})-(\ref{theta-nh0}), we obtain
\begin{align} \label{H1ta-00}
&\frac{d}{dt} \|\nabla_h \theta\|_2^2 + \frac{1}{Pe} \|\Delta_h \theta\|_2^2   \notag\\
&\leq C \|\nabla_h \mathbf u\|_2^2  \left(\|\mathbf u\|_2^2+\|\mathbf u_z\|_2^2 \right) \|\nabla_h \theta\|_2^2
 + \|\bar{w\theta}\|_2^2    \|\bar{\theta^2}\|_{L^{\infty}} +      \|\Delta_h w\|_2^2,
 \end{align}
 and
 \begin{align}   \label{lar-4}
 \|\nabla_h \theta(t)\|_2^2 + \frac{1}{Pe} \int_0^t  \|\Delta_h \theta\|_2^2 ds
 \leq  C(\|\theta_0\|_{H^1}, \|\nabla_h w_0\|_2, \|\omega_0\|_2), \;\; \text{for all} \,\, t\geq 0.
 \end{align}

Then, applying the uniform Gronwall Lemma on (\ref{H1ta-00}), we deduce
\begin{align*}
\|\nabla_h \theta(t+1)\|_2^2
&\leq e^{\int_t^{t+1} C \|\omega\|_2^2 (\|\mathbf u\|_2^2 +  \|\mathbf u_z\|_2^2) ds}
\int_t^{t+1}\left( \|\nabla_h \theta\|_2^2 +  \|\Delta_h w\|_2^2 +  \|\bar{w\theta}\|_2^2 \|\bar{\theta^2}\|_{L^{\infty}} \right) ds\notag\\
&\leq \left(e^{-\frac{2}{\gamma Pe}t} + e^{-\frac{1}{\kappa \gamma Re}t} \right)  C(\|\mathbf u_0\|_2,  \|w_0\|_2,  \|\theta_0\|_2, \|\partial_z \theta_0\|_2), \;\; \text{for all} \,\, t\geq 2,
\end{align*}
where we have used (\ref{large-2'}), (\ref{large-5}), (\ref{large-7})-(\ref{large-8}), (\ref{large-10}) and (\ref{theta-inf1}).

It remains to show that $\|\mathbf u_z\|_2^2 + \|w_z\|_2^2  +  \|\theta_z\|_2^2$ grows at most exponentially in time. Indeed, performing similar calculations as in section \ref{sec-z}, we obtain
\begin{align}     \label{end-1}
\|\partial_z w(t)\|_2^2 + \| \partial_z \mathbf u(t)\|_2^2 +       \| \partial_z \theta(t)\|_2^2  \leq \left( \|\partial_z w_0\|_2^2 + \| \partial_z \mathbf u_0\|_2^2 + C\| \partial_z \theta_0 \|_2^2   \right) e^{\mathcal M(t)},
\end{align}
where
\begin{align}   \label{end-2}
\mathcal  M(t)&=C\int_0^t  \Big(\|\Delta_h w\|_2^2 + \|\mathbf u_z\|_2^2  + \|\omega\|_2^2       \|\mathbf u\|_2^2      + \|\omega\|_2^2  \|\mathbf u_z\|_2^2    \notag\\
&   \hspace{0.5 in}  +  \|\Delta_h \theta\|_2^2   +   \|\bar{\theta^2}\|_{L^{\infty}}
 + \sup_{z\in [0,1]}  \int_0^t  \left|\bar{w \theta}\right|^2 d\tau  +1  \Big) ds     \notag\\
&\leq    C(\|\omega_0\|_2, \|\nabla_h w_0\|_2, \|\theta_0\|_{H^1}) +  C\left(\|\theta_0\|_2^2+  \|\partial_z \theta_0\|_2^2  +1 \right)t,
 \end{align}
for all $t\geq 0$, where we use (\ref{lar-1}), (\ref{lar-2}), (\ref{lar-3}), (\ref{theta-inf1}) and (\ref{lar-4}) to obtain the last inequality. Notice that (\ref{end-1})-(\ref{end-2}) implies (\ref{thm2-2}).
The proof for Theorem \ref{thm-decay} is complete.

\vspace{0.2 in}

\noindent {\bf Acknowledgment.} The work of E.S.T. was  supported in part by the Einstein Stiftung/Foundation - Berlin, through the Einstein Visiting Fellow Program and by the John Simon Guggenheim Memorial Foundation.

\vspace{0.1 in}

\bibliographystyle{amsplain}

\end{document}